\documentclass[11pt]{article}

\usepackage[utf8]{inputenc}
\usepackage[T1]{fontenc}
\usepackage{amsmath,amssymb,amsthm,mathtools}
\usepackage{graphicx}
\usepackage{booktabs}
\usepackage[round]{natbib}
\usepackage[margin=2.7cm]{geometry}
\usepackage{xcolor}
\usepackage[colorlinks=true,linkcolor=blue,citecolor=blue,urlcolor=blue]{hyperref}

\newtheorem{theorem}{Theorem}[section]
\newtheorem{lemma}[theorem]{Lemma}
\newtheorem{proposition}[theorem]{Proposition}
\newtheorem{corollary}[theorem]{Corollary}
\theoremstyle{definition}
\newtheorem{definition}[theorem]{Definition}
\newtheorem{assumption}[theorem]{Assumption}
\theoremstyle{remark}
\newtheorem{remark}[theorem]{Remark}

\title{Basin-Preserving Discretizations of Modern Hopfield Retrieval Dynamics:
Energy Cells, Dissipation, and the Attention Limit}
\author{Francisco R.~Villatoro\\[2pt]
\small Dept.\ de Lenguajes y Ciencias de la Computaci\'on, Universidad de M\'alaga, Spain\\
\small \texttt{frvillatoro@uma.es}\\
\small ORCID: \href{https://orcid.org/0000-0003-4314-6213}{0000-0003-4314-6213}}
\date{}

\begin{document}

\maketitle

\begin{abstract}
The retrieval dynamics of a modern Hopfield network is the gradient flow of a log-sum-exp energy, while the attention update is its exact difference-of-convex minimization step. We study which time discretizations preserve not only energy decay and equilibria but also basins of attraction. We introduce energy cells, connected components of sublevel sets containing one attractor and no other critical point. Our main theorem shows that every finite energy cell below the escape energy is contained simultaneously in the basin of the continuous flow, every relaxed attention map $\Psi_\theta=(1-\theta)\,\mathrm{id}+\theta\,\mathrm{attention}$ for $0<\theta<2$, and implicit Euler throughout its uniqueness regime. A parameter-uniform unit-curvature majorant yields unconditional dissipation and a monotone interpolation of each discrete step. We also derive explicit local contraction bounds near well-separated patterns, with a certified optimal slight overrelaxation; characterize proximal tunneling and overshoot beyond the preservation regimes; compare first-order error constants; establish an order barrier for scalar reparametrizations of the relaxed family; construct a second-order scalar-auxiliary-variable scheme; and extend cell preservation to damped difference-of-convex iterations in Bregman geometry, including a certified overrelaxed window under bounded asymmetry. Nine numerical campaigns test the bounds and failure mechanisms. In two-dimensional basin experiments, all observed disagreements between continuous and discrete retrieval occur above the attractor-specific numerically inferred escape level.
\end{abstract}


% ---------- body (Sections 1-8) -------------------------------------------
\section{Introduction}\label{sec:intro}

A dense associative memory, or modern Hopfield network, stores patterns $\xi_1,\xi_2,\dots,\xi_N \in \mathbb{R}^d$ and retrieves one of them from a corrupted query by descending the energy $E(x) = \tfrac12\|x\|^2 - \beta^{-1}\log\sum_\mu e^{\beta\xi_\mu^{\top}x}$ \citep{krotov2016,demircigil2017,ramsauer2021}. The exact difference-of-convex minimization step of this energy is $x \mapsto X\operatorname{softmax}(\beta X^{\top}x)$, the attention update of transformers, and the observation that attention is one such step \citep{ramsauer2021} has made these networks a shared object of the associative memory, optimization and deep learning communities. This paper looks at the same object from numerical analysis. Retrieval is the gradient flow $\dot x = -\nabla E(x)$; the iterative retrieval procedures studied here, including attention and its relaxed variants, admit time-discretization interpretations of that flow; and the question a memory poses to a discretization is not merely whether energy decays or equilibria survive, but whether the \emph{basins of attraction} of the stored patterns, the sets of queries from which each memory is recovered, are preserved.

Our thesis is that a well-defined part of each basin is preserved exactly, by the flow and by an entire family of discretizations at once, and that this common core is described by an object as elementary as the energy itself. An \emph{energy cell} of an attractor $x^\ast$ is the connected component of a sublevel set $\{E < c\}$ containing $x^\ast$; below the escape energy $c^\ast(x^\ast)$, the threshold up to which the component contains no critical point other than $x^\ast$, the cell contains no critical point but $x^\ast$. The central result, Theorem~\ref{thm:cells}, states that every such cell is contained simultaneously in the basin of the flow, in the basin of every member of the \emph{relaxed attention family} $\Psi_\theta(x) = (1-\theta)x + \theta X\operatorname{softmax}(\beta X^{\top}x)$ for all $\theta \in (0,2)$, and in the basin of the implicit Euler scheme throughout its uniqueness regime, with no condition on $\beta$, on the patterns, or on the step size beyond these ranges. Energy cells are common certified basin cores of the flow and its structure-preserving discretizations. Everything else in the paper either builds toward this statement, delimits it, or measures it.

The mechanism is a single inequality. The energy is the difference of the convex quadratic $\tfrac12\|x\|^2$ and the convex log-sum-exp, and this structure yields a global quadratic majorant of curvature exactly one, $E(y) \le E(x) + \langle\nabla E(x), y-x\rangle + \tfrac12\|y-x\|^2$, uniformly in $\beta$, $N$ and the pattern geometry (Lemma~\ref{lem:majorization}), even though the Hessian of $E$ is unbounded below as $\beta$ grows. Explicit Euler, semi-implicit convex splitting and exponential integration of the flow all collapse into the family $\Psi_\theta$ (Proposition~\ref{prop:family}), whose $\theta = 1$ member is attention and whose $\theta \to 1$ limits are those of the semi-implicit schemes as their step tends to infinity; the majorant gives unconditional dissipation $E(\Psi_\theta(x)) \le E(x) - \tfrac12\theta(2-\theta)\|\nabla E(x)\|^2$ on $\theta \in (0,2)$, exact up to the Bregman remainder of the log-sum-exp (Theorem~\ref{thm:dissipation}), and, evaluated along the segment between consecutive iterates, a continuous interpolant of every discrete orbit along which the energy never rises (Lemma~\ref{lem:interpolation}). Cells, being connected components of sublevel sets, cannot be left by such an interpolant, and precompactness together with the vanishing of the gradient identifies the unique limit; no analyticity or {\L}ojasiewicz argument is needed inside a cell.

The rest of the paper delimits and tests this core. Locally, near a pattern well separated from its competitors, we give explicit contraction factors for the whole family, exponentially small at $\theta = 1$ and minimized at a certified slight overrelaxation $\theta_\star = 2/(2-\ell)$, and unconditional contraction for the local branch of the implicit scheme, whose infinite-step limit is the deep-equilibrium formulation of retrieval rather than attention (Section~\ref{sec:local}); the certified cells are shown to lie between concentric balls with radii in the ratio $(1-\ell)^{-1/2}$ (Corollary~\ref{cor:sandwich}). Two propositions exhibit the complementary mechanisms by which preservation fails the \emph{global proximal selection} of the implicit Euler relation, which we distinguish carefully from its local branch, tunnels out of non-global cells beyond an explicit step threshold, at which point local minimizers cease to be its fixed points, and the explicit Euler map turns the attractor into a repeller beyond $\theta \approx 2$ (Section~\ref{sec:failure}). On accuracy and cost, an order barrier shows that no reparametrization of the family reaches second order under a verifiable non-degeneracy hypothesis, the exponential member is singled out by an error constant carried entirely by the softmax curvature, a second-order scalar-auxiliary-variable scheme is constructed with an exact modified-energy law at the cost of one attention evaluation per step and benchmarked against the classical discrete-gradient methods that dissipate the true energy exactly, and a closed-form analysis of the warm-started Picard hierarchy shows that no allocation of inner iterations certifies a better per-evaluation factor than attention (Sections~\ref{sec:order}--\ref{sec:cost}). A perspective section transfers the dissipation identity and the cell-preservation theorem to damped difference-of-convex iterations in Bregman geometry, where the overrelaxed window narrows to $\theta < 1 + 1/\kappa$ under an asymmetry bound $\kappa$ (Section~\ref{sec:generalized}). Nine numerical campaigns test each falsifiable prediction (Section~\ref{sec:experiments}); in particular, basins computed by high-precision integration of the flow, with critical points located by Newton iteration and continuous separatrices given by the stable manifolds of the saddles, disagree with the discrete basins at $12{,}727$ grid nodes across four values of $\theta$, and not one of them lies below the attractor-specific numerically inferred escape level used in that campaign.

The reader interested only in the core result may read Sections~\ref{sec:framework}, \ref{sec:globaldiss} and \ref{sec:basins}, which are self-contained.

\paragraph{Related work.}
The interpretation of memory retrieval in modern Hopfield networks as an optimization procedure is by now standard. \citet{ramsauer2021} derived the attention update as one step of the concave--convex procedure (CCCP) \citep{yuille2003} applied to the log-sum-exp energy, and subsequent work has sharpened this viewpoint. \citet{hu2023sparse} obtained sparse retrieval dynamics from a sparsemax energy via CCCP, connecting stationary points of the energy to fixed points of the dynamics through Zangwill's global convergence theory, and \citet{santos2024hfy} unified dense, sparse and structured variants through Fenchel--Young losses, proving exact one-step convergence to individual patterns under margin conditions. All of these analyses operate at the limit of exact minimization of the convex majorizer in each iteration, which in the terminology of the present paper corresponds to the limit $\Delta t \to \infty$ of a semi-implicit convex-splitting discretization \citep{eyre1998} of the retrieval flow $\dot{x} = -\nabla E(x)$. The intermediate regime of finite step sizes, the temporal accuracy of the resulting one-parameter family of schemes, and the effect of the discretization on the basins of attraction of the stored patterns have, to the best of our knowledge, not been analyzed.

Closest to our work is \citet{goemaere2024}, who accelerate the digital simulation of classical continuous Hopfield networks and hierarchical associative memories by recasting the equilibrium computation as a deep equilibrium model, solved with Anderson acceleration, and by an even-odd splitting that updates the two blocks of the bipartite layer graph alternately, roughly halving the iteration count. Their goal and ours are complementary rather than overlapping. They seek the fixed point as fast as possible, treating the trajectory as irrelevant, and their accelerated solvers explicitly forgo the energy-descent guarantee, which is assumed rather than proved to hold in practice; moreover, their analysis covers the classical energies of \citet{hopfield1984,krotov2021ham} and expressly excludes the log-sum-exp energies of modern Hopfield networks. In contrast, we regard the retrieval dynamics as a gradient flow to be integrated by structure-preserving methods, we establish unconditional energy dissipation and per-step contraction estimates for finite $\Delta t$, including the implicit Euler (proximal) scheme whose unique solvability holds under a one-sided Lipschitz restriction on the step, and we localize analytically where discrete and continuous basins can differ and quantify that discrepancy numerically in representative low-dimensional examples. The equivalence, noted by \citet{goemaere2024}, between damped Picard iteration on the equilibrium equation and the forward Euler method on the flow places their baseline scheme at the explicit end of the family of integrators studied here, while the attention update of \citet{ramsauer2021} sits at the opposite, fully implicit-in-the-convex-part end.

Our analytical tools descend from the nonlinear measure approach of \citet{qiao2001}, extended to implicit Euler discretizations of Hopfield networks in \citet{villatoro2005}, and are closely related to the recent non-Euclidean contraction theory of continuous-time neural networks \citep{davydov2025tac} and its application to the well-posedness of implicit networks \citep{jafarpour2021,davydov2024jmlr}. These works establish contractivity of the continuous-time flow, or existence and uniqueness of equilibria of implicit models, in weighted $\ell_1/\ell_\infty$ norms, and the preservation of strong contractivity under explicit and implicit Runge--Kutta discretizations has recently been characterized by \citet{kawano2026}. Our setting is complementary: retrieval operates in the multistable regime where global contractivity fails by design, and the object preserved by the discretization is not a contraction rate but the family of energy cells of Section~\ref{sec:basins}. On the optimization side, the damped difference-of-convex iterations underlying Section~\ref{sec:generalized} have been developed independently, with a Bregman-divergence reading of DCA by \citet{faust2023} and a continuous-time and damped-DCA theory by \citet{niu2026}; and the unstable fixed points that organize the basin boundaries above the escape energy of Section~\ref{sec:basins} have been shown to exist under natural geometric conditions by \citet{beise2026}. Iterated and damped retrieval updates have begun to appear on the modeling side: repeated Hopfield steps under local contraction collapse nearby queries onto a common fixed point \citep{sargolzaei2026}; damped Hopfield dynamics with monotone energy descent is used for query refinement on the hypersphere, with an explicitly acknowledged gap between the asymptotic theory and single-step deployment \citep{horen2026}; the drift of Langevin sampling on the Hopfield energy is an explicit Euler step of the flow \citep{alswaidan2026}; the iterated dynamics of trained Hopfield classifiers has been examined empirically \citep{mhndyn2026}; and basin-invariance results have been proved for the discretization of a different, measure-valued retrieval flow of spherical Hellinger--Kantorovich type \citep{mustafi2026}. Convergence of the exact CCCP step has also been established for a softmin variant of the modern Hopfield energy \citep{bao2026}, a different energy from \eqref{eq:energy}. None of these develops a numerical-analysis theory of the standard log-sum-exp flow. To our knowledge, no prior deterministic treatment of that flow establishes a common sublevel-component basin core simultaneously for the continuous dynamics and a family of relaxed and implicit discretizations, with quantitative certificates and failure thresholds, which is the contribution of Sections~\ref{sec:dissipation}--\ref{sec:basins}. Finally, our setting should be distinguished from the interacting particle description of transformers \citep{geshkovski2025}, which studies the coupled evolution of all tokens and their long-time clustering, whereas we study the retrieval of a single query against a fixed set of stored memories, the regime in which associative memory guarantees are formulated \citep{ramsauer2021,hoover2023}.

% Section 2 -- draft v1
% Assumes: amsmath, amsthm, natbib; theorem environments: lemma, proposition, corollary, remark, definition (numbered within section).
% Notation conventions: patterns \xi_\mu, memory matrix X, query/state x, inverse temperature \beta.

\section{A one-parameter family of time integrators for memory retrieval}\label{sec:framework}

\subsection{The retrieval flow and its energy}\label{sec:flow}

Let $\xi_1,\dots,\xi_N \in \mathbb{R}^d$ denote the stored patterns, collected as the columns of the memory matrix $X = (\xi_1,\dots,\xi_N) \in \mathbb{R}^{d\times N}$, and let $\beta > 0$ be the inverse temperature. The modern Hopfield network of \citet{ramsauer2021}, building on the dense associative memories of \citet{krotov2016} and \citet{demircigil2017}, retrieves a stored pattern from a query by descending the energy
\begin{equation}\label{eq:energy}
E(x) \;=\; \tfrac{1}{2}\,\|x\|_2^2 \;-\; \operatorname{lse}_\beta\!\big(X^{\top}x\big),
\qquad
\operatorname{lse}_\beta(z) \;=\; \beta^{-1}\log \sum_{\mu=1}^{N} e^{\beta z_\mu},
\end{equation}
where we omit the additive constants of \citet{ramsauer2021}, which play no role in the dynamics. Throughout, $\|\cdot\|$ denotes the Euclidean norm, $\sigma_{\max}(X)$ the largest singular value of $X$, and $M = \max_{\mu} \|\xi_\mu\|$. The gradient of \eqref{eq:energy} is
\begin{equation}\label{eq:gradient}
\nabla E(x) \;=\; x - X\,p(x),
\qquad
p(x) \;=\; \operatorname{softmax}\!\big(\beta X^{\top} x\big) \in \Delta^{N-1},
\end{equation}
with $\Delta^{N-1}$ the probability simplex, and the retrieval dynamics is the gradient flow
\begin{equation}\label{eq:flow}
\dot{x} \;=\; -\nabla E(x) \;=\; -x + X\,p(x),
\end{equation}
formally a Hopfield equation $\dot{x} = -x + f(x)$ in which the nonlinearity $f = X p(\cdot)$ couples all components of the state through the softmax, in contrast with the componentwise transfer functions of the classical model \citep{hopfield1984}. The following lemma collects the structural properties of \eqref{eq:energy} on which the whole paper rests.

\begin{lemma}[Structure of the energy]\label{lem:structure}
Let $E$ be given by \eqref{eq:energy}. Then the following hold.
\begin{enumerate}
\item[\textup{(a)}] \textup{(DC structure)} $E = E_{+} - E_{-}$ with $E_{+}(x) = \tfrac12\|x\|^2$ and $E_{-}(x) = \operatorname{lse}_\beta(X^{\top}x)$ both convex, and $E_{-}$ has gradient $\nabla E_{-}(x) = X p(x) \in \operatorname{conv}\{\xi_1,\dots,\xi_N\}$.
\item[\textup{(b)}] \textup{(Curvature bounds)} For every $x \in \mathbb{R}^d$,
\begin{equation}\label{eq:hessianbounds}
\Big(1 - \tfrac{\beta}{2}\,\sigma_{\max}^2(X)\Big) I \;\preceq\; \nabla^2 E(x) \;\preceq\; I .
\end{equation}
\item[\textup{(c)}] \textup{(One-sided Lipschitz constant)} The vector field $F = -\nabla E$ of \eqref{eq:flow} satisfies $\langle F(x)-F(y),\,x-y\rangle \le \nu\,\|x-y\|^2$ for all $x,y$, with $\nu = \tfrac{\beta}{2}\,\sigma_{\max}^2(X) - 1$.
\item[\textup{(d)}] \textup{(Coercivity and location of equilibria)} $E(x) \ge \tfrac12\|x\|^2 - M\|x\| - \beta^{-1}\log N$, so $E$ is coercive; equilibria of \eqref{eq:flow} exist, satisfy $x^\ast = X p(x^\ast)$, and therefore lie in $\operatorname{conv}\{\xi_1,\dots,\xi_N\}$, in particular $\|x^\ast\| \le M$.
\end{enumerate}
\end{lemma}

\begin{proof}
Part (a) is immediate since $\operatorname{lse}_\beta$ is convex and nondecreasing in each argument and $X^{\top}x$ is linear, while $\nabla E_{-}(x) = Xp(x)$ is a convex combination of the columns of $X$. For (b), differentiation of \eqref{eq:gradient} gives $\nabla^2E(x) = I - \beta\, X H(p) X^{\top}$ with $H(p) = \operatorname{diag}(p) - p\,p^{\top} \succeq 0$, which yields the upper bound. For the lower bound, note that for any unit vector $v \in \mathbb{R}^N$ the quadratic form $v^{\top}H(p)\,v = \sum_\mu p_\mu v_\mu^2 - \big(\sum_\mu p_\mu v_\mu\big)^2$ is the variance of a random variable taking the value $v_\mu$ with probability $p_\mu$; by Popoviciu's inequality this variance is at most $\tfrac14(\max_\mu v_\mu - \min_\mu v_\mu)^2 \le \tfrac12\|v\|^2 = \tfrac12$, hence $H(p) \preceq \tfrac12 I$ and $\beta X H(p) X^{\top} \preceq \tfrac{\beta}{2}\sigma_{\max}^2(X)\, I$. Part (c) follows from the lower bound in \eqref{eq:hessianbounds} by the mean value theorem applied to $t \mapsto \langle \nabla E(y + t(x-y)), x - y\rangle$. Part (d) follows from $\operatorname{lse}_\beta(X^{\top}x) \le \max_\mu \xi_\mu^{\top}x + \beta^{-1}\log N \le M\|x\| + \beta^{-1}\log N$, existence of a global minimizer by coercivity, and the fixed point form of $\nabla E(x^\ast) = 0$.
\end{proof}

Two consequences of Lemma~\ref{lem:structure} shape the analysis. First, when $\beta\,\sigma_{\max}^2(X) < 2$ the constant $\nu$ in (c) is negative, the flow \eqref{eq:flow} is a strict contraction in the Euclidean norm, $E$ is strongly convex, and the equilibrium is unique; useful memory retrieval requires leaving this regime, so $\beta\,\sigma_{\max}^2(X) \ge 2$ is \emph{necessary} for multistability. It is not sufficient, since for $N = 1$ the log-sum-exp term is affine and $E$ is strongly convex for every $\beta$, however large $\beta\,\sigma_{\max}^2(X)$ may be. Throughout the paper the multistable regime is therefore certified operationally, through the retrieval condition \eqref{eq:retrieval-cond} of Section~\ref{sec:local} holding at two or more patterns, rather than through the spectral bound; every stability statement in that regime is necessarily local and metastable. This inverts the emphasis of the classical stability literature for discrete-time Hopfield networks \citep{qiao2001,villatoro2005}, where global exponential stability is the desired property; here globally contracting dynamics correspond to a memory that has collapsed to a single spurious attractor. Second, the upper bound $\nabla^2E \preceq I$ holds uniformly in $\beta$, a fact that is far from obvious at first sight, since the Lipschitz constant of $\nabla E$ grows linearly in $\beta$ through the lower bound in \eqref{eq:hessianbounds}. The asymmetry between the two bounds, bounded curvature from above and unbounded from below, is precisely the difference-of-convex structure of (a), and it is what allows unconditional energy dissipation results for semi-implicit schemes in Section~\ref{sec:dissipation}.

\subsection{Four elementary time discretizations}\label{sec:schemes}

Fix a step size $\Delta t > 0$ and denote by $x^k \approx x(k\,\Delta t)$ the discrete trajectory. We consider four one-step discretizations of \eqref{eq:flow}, chosen because each occupies a distinguished position with respect to the DC structure of Lemma~\ref{lem:structure}(a).

\begin{definition}[Basic schemes]\label{def:schemes}
The \emph{explicit Euler} (EE), \emph{implicit Euler} (IE), \emph{convex splitting} (CS) and \emph{exponential} (ETD) schemes are given, respectively, by
\begin{align}
\textup{(EE)}\quad & x^{k+1} = x^{k} - \Delta t\,\nabla E(x^{k}) = (1-\Delta t)\,x^{k} + \Delta t\, X p(x^{k}), \label{eq:ee}\\[2pt]
\textup{(IE)}\quad & x^{k+1} = x^{k} - \Delta t\,\nabla E(x^{k+1}), \label{eq:ie}\\[2pt]
\textup{(CS)}\quad & \frac{x^{k+1}-x^{k}}{\Delta t} = -\nabla E_{+}(x^{k+1}) + \nabla E_{-}(x^{k}), \quad\text{i.e.}\quad x^{k+1} = \frac{x^{k} + \Delta t\, X p(x^{k})}{1+\Delta t}, \label{eq:cs}\\[2pt]
\textup{(ETD)}\quad & x^{k+1} = e^{-\Delta t}\,x^{k} + \big(1 - e^{-\Delta t}\big)\, X p(x^{k}), \label{eq:etd}
\end{align}
where \eqref{eq:cs} treats the convex part of the energy implicitly and the concave part explicitly in the sense of \citet{eyre1998}, and \eqref{eq:etd} is the first-order exponential integrator applied to the semilinear form $\dot{x} = -x + Xp(x)$ of the flow.
\end{definition}

\begin{definition}[Implicit Euler objects]\label{def:ie-objects}
For a step size $\Delta t>0$, the \emph{implicit Euler relation} at $x\in\mathbb{R}^d$ is the equation
\[
 y=x-\Delta t\,\nabla E(y)
\]
in the unknown $y$. The associated \emph{global proximal map} is the generally set-valued map
\[
 \mathcal{P}_{\Delta t}(x)
 :=
 \operatorname*{argmin}_{y\in\mathbb{R}^d}
 \left\{E(y)+\frac{1}{2\Delta t}\|y-x\|^2\right\}.
\]
It is nonempty for every $x$ and every $\Delta t>0$ by coercivity of $E$, and every $y\in\mathcal{P}_{\Delta t}(x)$ solves the implicit Euler relation. Whenever a single-valued global proximal orbit is considered, $P_{\Delta t}$ denotes an arbitrary selection satisfying $P_{\Delta t}(x)\in\mathcal{P}_{\Delta t}(x)$; statements about such a selection hold for every choice unless stated otherwise. The \emph{local branch}, denoted $\Phi_{\Delta t}$ and defined on retrieval balls in Section~\ref{sec:local}, assigns to $x$ the unique solution of the relation lying in the same retrieval ball. If $\Delta t\,\nu<1$, the proximal functional is strongly convex, the implicit Euler relation has exactly one solution, and
\[
 \mathcal{P}_{\Delta t}(x)=\{P_{\Delta t}(x)\}=\{\Phi_{\Delta t}(x)\}
\]
wherever the local branch is defined. Outside this uniqueness regime the distinction is substantive. The local branch may continue to retrieve a designated memory while every global proximal minimizer may tunnel to a lower-energy basin; see Proposition~\ref{prop:tunneling}.
\end{definition}

The relation requires solving a nonlinear system at each step, and the well-posedness of its selections follows the pattern familiar from the classical discrete-time Hopfield literature \citep{atencia2002,villatoro2005}. Since $E$ is coercive by Lemma~\ref{lem:structure}(d), the proximal functional $y \mapsto E(y) + \tfrac{1}{2\Delta t}\|y - x^{k}\|^2$ attains its global minimum for every $\Delta t > 0$, and every global minimizer satisfies \eqref{eq:ie}; the step is thus well defined as the set-valued update $x^{k+1} \in \mathcal{P}_{\Delta t}(x^{k})$ for arbitrary step sizes. If moreover $\Delta t\,\nu < 1$, with $\nu$ the one-sided Lipschitz constant of Lemma~\ref{lem:structure}(c), then the proximal functional is $(\tfrac{1}{\Delta t}-\nu)$-strongly convex, the step is single valued, and \eqref{eq:ie} has exactly one solution. When $\nu>0$, the threshold $\Delta t < 1/\nu$ is the exact analogue, for the energy \eqref{eq:energy}, of the step size restriction under which the implicit Euler discretization of the classical Hopfield network admits a unique solution and inherits the Lyapunov function of the continuous model \citep{atencia2002}, and of the uniqueness condition in the nonlinear measure analysis of \citet{villatoro2005}; when $\nu\le0$, uniqueness holds for every $\Delta t>0$. The practical cost of the implicit step is deferred to Section~\ref{sec:cost}; we only note here that the Hessian of the proximal functional is the identity plus a correction of rank at most $N$, so that Newton iterations are inexpensive whenever $N \ll d$ or the Woodbury identity applies.

\subsection{The damped attention family and the limit $\Delta t \to \infty$}\label{sec:family}

The three explicit-in-$E_-$ schemes of Definition~\ref{def:schemes} collapse, for the specific energy \eqref{eq:energy}, into a single one-parameter family of maps. Define, for $\theta > 0$,
\begin{equation}\label{eq:family}
\Psi_\theta(x) \;=\; (1-\theta)\,x + \theta\, X p(x) \;=\; x - \theta\,\nabla E(x).
\end{equation}

\begin{proposition}[Reparametrization and the attention limit]\label{prop:family}
For the energy \eqref{eq:energy}, the maps \eqref{eq:ee}, \eqref{eq:cs} and \eqref{eq:etd} coincide with $\Psi_\theta$ under the substitutions
\[
\theta_{\mathrm{EE}} = \Delta t, \qquad
\theta_{\mathrm{CS}} = \frac{\Delta t}{1+\Delta t}, \qquad
\theta_{\mathrm{ETD}} = 1 - e^{-\Delta t}.
\]
The maps $\Delta t \mapsto \theta_{\mathrm{CS}}$ and $\Delta t \mapsto \theta_{\mathrm{ETD}}$ are strictly increasing bijections from $(0,\infty)$ onto $(0,1)$; consequently the CS and ETD schemes generate identical discrete orbits up to a relabeling of the step size, both are underrelaxations of the update rule of \citet{ramsauer2021}, and both satisfy
\[
\lim_{\Delta t \to \infty} x^{k+1} \;=\; \Psi_1(x^{k}) \;=\; X\operatorname{softmax}\!\big(\beta X^{\top}x^{k}\big),
\]
which is the attention update. Moreover, for every $\theta > 0$ the fixed points of $\Psi_\theta$, and for every $\Delta t > 0$ the fixed points of the implicit Euler \emph{relation} of Definition~\ref{def:ie-objects}, coincide exactly with the critical points of $E$. The set-valued fixed points $x \in \mathcal{P}_{\Delta t}(x)$ are instead the $\Delta t$-proximally stationary points $\{x : E(z) \ge E(x) - \tfrac{1}{2\Delta t}\|z-x\|^2 \text{ for all } z\}$. These sets are nested decreasing as $\Delta t$ increases, contain every critical point when $\Delta t\,\nu < 1$, and their intersection over all $\Delta t>0$ is exactly the set of global minimizers of $E$; see Proposition~\ref{prop:tunneling}.
\end{proposition}

\begin{proof}
The identities for $\theta_{\mathrm{EE}}$, $\theta_{\mathrm{CS}}$ and $\theta_{\mathrm{ETD}}$ are elementary algebra on \eqref{eq:ee}, \eqref{eq:cs} and \eqref{eq:etd}, using the second expression in \eqref{eq:family}. Monotonicity and the ranges of the substitutions are immediate, as is the limit $\theta \to 1$ as $\Delta t \to \infty$. For the fixed points, $\Psi_\theta(x) = x$ with $\theta \ne 0$ is equivalent to $\nabla E(x) = 0$, and $x = x - \Delta t\,\nabla E(x)$ likewise; the characterization of the points with $x \in \mathcal{P}_{\Delta t}(x)$ is immediate from the definition of a global minimizer of the proximal functional; nestedness follows because the quadratic penalty decreases with $\Delta t$, and intersecting the resulting inequalities over arbitrarily large $\Delta t$ leaves precisely $E(z)\ge E(x)$ for all $z$.
\end{proof}

Proposition~\ref{prop:family} deserves three comments. First, it identifies the attention update, which is simultaneously one exact step of the concave--convex procedure on \eqref{eq:energy} \citep{yuille2003,ramsauer2021} and one gradient descent step of unit length \citep{incontextdenoising2025}, as the extreme member $\theta = 1$ of a family of underrelaxed integrators, with the semi-implicit convex splitting of \citet{eyre1998} and the exponential integrator as the interior members and the explicit Euler method as the only member able to leave the interval $\theta \in (0,1]$. The customary choice $\Delta t = 1$ in the neural network literature, which for the classical Hopfield model produces the Takeda--Goodman iteration and its well-documented instabilities \citep{takeda1986}, here corresponds to $\theta_{\mathrm{EE}} = 1$ and lands exactly on the attention update, which is stable for a structural reason made precise below; the coincidence is specific to the unit relaxation time in \eqref{eq:flow} and dissolves for any other time constant. Second, since the CS and ETD orbits coincide as sets, the distinction between the two schemes is invisible to any question that concerns only the equilibrium and the basin from which it is reached, and becomes meaningful only when the discrete trajectory is required to approximate the continuous one, a distinction we take up in the accuracy analysis of Section~\ref{sec:order}. Third, and as a matter of scope, the collapse of three schemes into one family is special to the energy \eqref{eq:energy}, whose convex part is exactly $\tfrac12\|x\|^2$ with linear gradient; for the generalized energies of the sparse and Fenchel--Young models \citep{hu2023sparse,santos2024hfy}, or for spherical models in which $E_{+}$ incorporates normalization constraints \citep{hoover2023}, the convex part is no longer quadratic, the implicit-in-$E_+$ step is no longer explicit, and the four schemes of Definition~\ref{def:schemes} are genuinely distinct. The present section should therefore be read as the base case of the theory, with the extensions treated in Section~\ref{sec:generalized}.

The reason the whole family $\theta \in (0,1]$, including the attention endpoint, dissipates the energy unconditionally is the following global majorization, which replaces the classical descent lemma and requires no Lipschitz constant.

\begin{lemma}[Unit-curvature majorization]\label{lem:majorization}
For all $x, y \in \mathbb{R}^d$,
\begin{equation}\label{eq:majorization}
E(y) \;\le\; Q_x(y) \;:=\; E(x) + \big\langle \nabla E(x),\, y - x \big\rangle + \tfrac12\,\|y - x\|^2,
\end{equation}
with equality at $y = x$. The majorizer $Q_x$ is a $1$-strongly convex quadratic whose unique minimizer is $x - \nabla E(x) = X p(x) = \Psi_1(x)$, the attention update of $x$.
\end{lemma}

\begin{proof}
By convexity of $E_{-}$, $-E_{-}(y) \le -E_{-}(x) - \langle \nabla E_{-}(x), y-x\rangle$, and expanding $E_{+}(y) = \tfrac12\|y\|^2 = \tfrac12\|x\|^2 + \langle x, y-x\rangle + \tfrac12\|y-x\|^2$ gives \eqref{eq:majorization} after collecting terms. The minimizer of $Q_x$ is $y = x - \nabla E(x)$ by first-order optimality.
\end{proof}

Lemma~\ref{lem:majorization} states that the energy \eqref{eq:energy} admits a quadratic upper bound of curvature exactly one at every point, uniformly in $\beta$, $N$ and the pattern geometry, even though its curvature from below degrades linearly in $\beta$ by \eqref{eq:hessianbounds}. In optimization language, gradient descent on $E$ enjoys the descent guarantee of a $1$-smooth function despite $E$ not being $1$-smooth, because all the unbounded curvature sits in the concave part, which the linearization in \eqref{eq:majorization} majorizes for free. Every dissipation estimate in Section~\ref{sec:dissipation} is obtained by evaluating \eqref{eq:majorization} along the maps $\Psi_\theta$ and $\operatorname{prox}_{\Delta t E}$; in particular it yields, for all $\theta \in (0,2)$ and without any restriction on $\beta$, $X$ or the starting point, the per-step decrease $E(\Psi_\theta(x)) \le E(x) - \tfrac12\,\theta(2-\theta)\,\|\nabla E(x)\|^2$, whose proof and consequences, including the corresponding statements for the implicit scheme and the resulting a~priori iteration counts, are the subject of the next section.

% --- end of Section 2 ---

% Section 3 -- draft v1
% Assumes: amsmath, amsthm, natbib; environments theorem, lemma, proposition, corollary, remark numbered within section.
% Depends on Section 2: eq:energy, eq:gradient, eq:flow, eq:family, eq:ie, lem:structure, lem:majorization, prop:family.

\section{Energy dissipation and convergence of the basic schemes}\label{sec:dissipation}

This section establishes the results announced at the end of Section~\ref{sec:framework}: unconditional energy dissipation for the damped attention family $\Psi_\theta$ and for the implicit Euler scheme, with explicit per-step decrease and a priori stationarity rates (Section~\ref{sec:globaldiss}); convergence of the full sequence of iterates to a single critical point of $E$, a consequence of the real analyticity of the energy (Section~\ref{sec:iterates}); and local geometric convergence near well-separated patterns, with contraction factors that are explicit in the step size, the inverse temperature and the pattern separation (Section~\ref{sec:local}). Throughout, $E$, $X$, $p$, $M$, $\sigma_{\max}(X)$, $\nu$ and $\Psi_\theta$ are as in Section~\ref{sec:framework}.

\subsection{Unconditional dissipation and stationarity rates}\label{sec:globaldiss}

We first record that all schemes under consideration confine the iterates to an explicit ball, so that no growth condition needs to be imposed a posteriori.

\begin{lemma}[Absorbing ball]\label{lem:absorbing}
Let $R_\theta = \theta M / (1-|1-\theta|)$ for $\theta \in (0,2)$, so that $R_\theta = M$ for $\theta \in (0,1]$. Then $\|\Psi_\theta(x)\| \le |1-\theta|\,\|x\| + \theta M$ for all $x$, the ball $\bar{B}(0, \max\{\|x^0\|, R_\theta\})$ is positively invariant for $\Psi_\theta$, and $\limsup_{k} \|x^k\| \le R_\theta$ with geometric rate $|1-\theta|$. The same conclusions hold, with $R = M$, for any sequence of solutions of the implicit Euler relation \eqref{eq:ie}, under any selection of Definition~\ref{def:ie-objects}.
\end{lemma}

\begin{proof}
Since $Xp(x) \in \operatorname{conv}\{\xi_\mu\}$, $\|Xp(x)\| \le M$, and the bound for $\Psi_\theta$ follows from \eqref{eq:family} by the triangle inequality. For the implicit scheme, rearranging \eqref{eq:ie} for the energy \eqref{eq:energy} gives $x^{k+1} = \big(x^{k} + \Delta t\, Xp(x^{k+1})\big)/(1+\Delta t)$, whence $\|x^{k+1}\| \le (\|x^{k}\| + \Delta t\,M)/(1+\Delta t) \le \max\{\|x^{k}\|, M\}$; the computation uses only the rearranged form of the relation, hence holds for every selection.
\end{proof}

\begin{theorem}[Unconditional energy dissipation]\label{thm:dissipation}
Let $E$ be the energy \eqref{eq:energy}. Then the following hold for every $\beta > 0$, every pattern matrix $X$ and every starting point $x^0 \in \mathbb{R}^d$, with no restriction on the step size beyond the stated ranges.
\begin{enumerate}
\item[\textup{(a)}] For every $\theta \in (0,2)$, the iterates $x^{k+1} = \Psi_\theta(x^{k})$ satisfy
\begin{equation}\label{eq:descent-family}
E(x^{k+1}) \;\le\; E(x^{k}) \;-\; \tfrac12\,\theta\,(2-\theta)\,\big\|\nabla E(x^{k})\big\|^2 .
\end{equation}
In particular the attention update $\theta = 1$ dissipates at least $\tfrac12 \|\nabla E(x^{k})\|^2$ per step.
\item[\textup{(b)}] For every $\Delta t > 0$, any sequence with $x^{k+1}\in\mathcal{P}_{\Delta t}(x^k)$ satisfies \eqref{eq:ie} together with
\begin{equation}\label{eq:descent-ie}
E(x^{k+1}) \;\le\; E(x^{k}) \;-\; \tfrac{1}{2\Delta t}\,\big\|x^{k+1} - x^{k}\big\|^2 \;=\; E(x^{k}) \;-\; \tfrac{\Delta t}{2}\,\big\|\nabla E(x^{k+1})\big\|^2 .
\end{equation}
\end{enumerate}
In both cases the sequence $E(x^{k})$ is nonincreasing and converges, and $\nabla E(x^{k}) \to 0$.
\end{theorem}

\begin{proof}
For (a), evaluate the majorization \eqref{eq:majorization} at $y = \Psi_\theta(x^{k}) = x^{k} - \theta \nabla E(x^{k})$: the linear term contributes $-\theta \|\nabla E(x^{k})\|^2$ and the quadratic term $\tfrac12 \theta^2 \|\nabla E(x^{k})\|^2$, giving \eqref{eq:descent-family}. For (b), minimality against the competitor $y = x^{k}$ gives $E(x^{k+1}) + \tfrac{1}{2\Delta t}\|x^{k+1}-x^{k}\|^2 \le E(x^{k})$, and the first-order optimality condition of the proximal functional is exactly \eqref{eq:ie}, which converts $\|x^{k+1}-x^{k}\| = \Delta t\, \|\nabla E(x^{k+1})\|$. In both cases $E(x^{k})$ is nonincreasing and bounded below by Lemma~\ref{lem:structure}(d), hence convergent, and summing the per-step decreases forces the gradient terms to vanish.
\end{proof}

The inequality in part (a) is in fact an identity up to the curvature of the concave part: writing $D_{-}(y,x) = E_{-}(y) - E_{-}(x) - \langle\nabla E_{-}(x),\, y-x\rangle \ge 0$ for the Bregman remainder of $E_{-} = \operatorname{lse}_\beta(X^{\top}\cdot\,)$, one has $E(x^{k}) - E(x^{k+1}) = \tfrac12\theta(2-\theta)\|\nabla E(x^{k})\|^2 + D_{-}(x^{k+1}, x^{k})$, a special case of the exact Bregman identity of Theorem~\ref{thm:bregman-descent}; the slack of the certified decrease is exactly the curvature of the log-sum-exp along the step, which is the structural explanation of the near-unit ratios observed in Campaign~A of Section~\ref{sec:experiments}.

We stress the division of labor in part (b): existence of the implicit step for arbitrary $\Delta t$ is guaranteed by coercivity through global minimization of the proximal functional, and dissipation holds for any such global minimizer, whereas uniqueness of the step requires the restriction $\Delta t\, \nu < 1$ of Section~\ref{sec:schemes}. This is the same division as in the classical discrete-time Hopfield setting, where the implicit Euler network of \citet{atencia2002} inherits the continuous Lyapunov function precisely below the inverse of the one-sided Lipschitz constant, and where the nonlinear measure analysis of \citet{villatoro2005} certifies uniqueness and exponential decay in the same regime; the novelty here is that for the energy \eqref{eq:energy} the dissipation inequality itself survives, in the set-valued sense, for every step size.

Summing \eqref{eq:descent-family} and \eqref{eq:descent-ie} and using the explicit bounds of Lemma~\ref{lem:structure}(d) yields fully explicit a priori stationarity rates.

\begin{corollary}[A priori stationarity rates]\label{cor:rates}
Let $G_0 = E(x^0) - \inf E$. Then $\inf E \ge -\tfrac12 M^2 - \beta^{-1}\log N$ and $E(x^0) \le \tfrac12\|x^0\|^2 + M\|x^0\|$, so that $G_0 \le \tfrac12\big(\|x^0\| + M\big)^2 + \beta^{-1}\log N$, and for every $K \ge 1$,
\[
\min_{0 \le k < K} \big\|\nabla E(x^{k})\big\|^2 \;\le\; \frac{2\,G_0}{\theta(2-\theta)\,K}
\quad\text{for } x^{k+1} = \Psi_\theta(x^{k}),
\]
\[
\min_{1 \le k \le K} \big\|\nabla E(x^{k})\big\|^2 \;\le\; \frac{2\,G_0}{\Delta t\, K}
\quad\text{for any sequence }x^{k+1}\in\mathcal{P}_{\Delta t}(x^k).
\]
\end{corollary}

The constant $\theta(2-\theta)$ is maximized at $\theta = 1$; within the damped attention family, and by this criterion, the attention update is the optimal member, and underrelaxation can only slow the worst-case decrease. Locally the picture is finer; by Theorem~\ref{thm:overrelax} the certified contraction factor near a well-separated attractor is minimized not at $\theta = 1$ but at the slight overrelaxation $\theta_\star = 2/(2-\ell_\mu(r))$. The implicit rate improves without bound as $\Delta t$ grows, consistently with the observation, made precise in Section~\ref{sec:local}, that the $\Delta t \to \infty$ limit of the implicit scheme is not the attention update but the exact equilibrium computation.

\subsection{Convergence of the iterates}\label{sec:iterates}

Theorem~\ref{thm:dissipation} yields $\nabla E(x^{k}) \to 0$ and, with Lemma~\ref{lem:absorbing}, that the set of accumulation points is a nonempty connected subset of the critical set of $E$. For a nonconvex energy this does not by itself exclude drift along a continuum of critical points. The energy \eqref{eq:energy} is, however, real analytic on $\mathbb{R}^d$, being the difference of a quadratic and the composition of $\operatorname{lse}_\beta$ with a linear map, and therefore satisfies the {\L}ojasiewicz gradient inequality at every point. Convergence of the full sequence then follows from the general theory of descent methods for analytic cost functions.

\begin{theorem}[Convergence to a single critical point]\label{thm:single-limit}
For every $\theta \in (0,2)$ and every $x^0$, the sequence $x^{k+1} = \Psi_\theta(x^{k})$ converges to a single critical point $x^\ast$ of $E$. The same holds for every sequence $x^{k+1}\in\mathcal{P}_{\Delta t}(x^k)$ of Theorem~\textup{\ref{thm:dissipation}(b)} for every $\Delta t > 0$.
\end{theorem}

\begin{proof}
The iterates are bounded by Lemma~\ref{lem:absorbing}. For the family, the step satisfies the identity $x^{k+1} - x^{k} = -\theta\, \nabla E(x^{k})$, so \eqref{eq:descent-family} gives the strong descent condition $E(x^{k}) - E(x^{k+1}) \ge \tfrac{2-\theta}{2}\, \|\nabla E(x^{k})\|\,\|x^{k+1}-x^{k}\|$, and $x^{k+1} = x^{k}$ occurs only at critical points; these are the hypotheses of the convergence theorem for descent methods on real analytic functions of \citet{absil2005}. For the implicit scheme, \eqref{eq:descent-ie} is the sufficient decrease condition and the optimality condition $\nabla E(x^{k+1}) = -(x^{k+1}-x^{k})/\Delta t$ is the relative error condition, with constants $a = 1/(2\Delta t)$ and $b = 1/\Delta t$, of the abstract convergence framework of \citet{attouch2013} for Kurdyka--{\L}ojasiewicz functions, a class that contains all real analytic functions. For the local branch on a retrieval ball, convergence is immediate from the contraction of Theorem~\ref{thm:local-ie} and requires no {\L}ojasiewicz argument.
\end{proof}

\begin{remark}
{\L}ojasiewicz exponents also convert Theorem~\ref{thm:single-limit} into local convergence rates, geometric or algebraic according to the exponent at the limit point. We do not pursue this route because near the physically relevant limit points, namely the retrieval attractors associated with well-separated patterns, the energy is strongly convex by the estimates of the next subsection, the exponent is $\tfrac12$, and sharper and fully explicit geometric rates are available directly.
\end{remark}

\subsection{Local geometric convergence near well-separated patterns}\label{sec:local}

We now quantify the metastable regime. Following \citet{ramsauer2021}, the relevant geometric quantity is the separation of a pattern from its competitors,
\begin{equation}\label{eq:separation}
\Delta_\mu \;=\; \min_{\nu \ne \mu}\, \big(\xi_\mu - \xi_\nu\big)^{\top} \xi_\mu ,
\end{equation}
and the margin $\Delta_\mu$ in \eqref{eq:separation} controls the exponential concentration of retrieval. Retrieval statements hold on balls $\bar{B}(\xi_\mu, r)$ on which the softmax concentrates on the $\mu$-th coordinate. The next lemma makes the concentration quantitative in the two forms needed below, for the value of the map and for its differential.

\begin{lemma}[Concentration on a retrieval ball]\label{lem:concentration}
Let $r > 0$ satisfy $\Delta_\mu > 2Mr$ and set
\[
\varepsilon_\mu(r) \;=\; (N-1)\, e^{-\beta\,(\Delta_\mu - 2Mr)} .
\]
Then for every $x \in \bar{B}(\xi_\mu, r)$: \textup{(i)} $1 - p_\mu(x) \le \varepsilon_\mu(r)$; \textup{(ii)} $\|X p(x) - \xi_\mu\| \le 2M\,\varepsilon_\mu(r)$; \textup{(iii)} $\big\|\nabla^2 E(x) - I\big\| = \beta\,\big\|X H(p(x)) X^{\top}\big\| \le \ell_\mu(r) := 2\,\beta\,\sigma_{\max}^2(X)\,\varepsilon_\mu(r)$, where $H(p) = \operatorname{diag}(p) - pp^{\top}$ as in Lemma~\textup{\ref{lem:structure}}.
\end{lemma}

\begin{proof}
For (i), for $\nu \ne \mu$ and $x \in \bar{B}(\xi_\mu, r)$, $(\xi_\nu - \xi_\mu)^{\top}x = (\xi_\nu - \xi_\mu)^{\top}\xi_\mu + (\xi_\nu - \xi_\mu)^{\top}(x - \xi_\mu) \le -\Delta_\mu + 2Mr$, hence $p_\nu(x) \le p_\nu(x)/p_\mu(x) = e^{\beta(\xi_\nu-\xi_\mu)^{\top}x} \le e^{-\beta(\Delta_\mu - 2Mr)}$, and summing over $\nu \ne \mu$ gives (i). For (ii), $Xp(x) - \xi_\mu = \sum_{\nu} p_\nu (\xi_\nu - \xi_\mu)$ and $\|\xi_\nu - \xi_\mu\| \le 2M$, so the norm is at most $2M(1-p_\mu)$. For (iii), for any unit vector $v \in \mathbb{R}^N$, $v^{\top} H(p) v = \operatorname{Var}_p(v) \le \mathbb{E}_p\big[(v - v_\mu)^2\big] = \sum_{\nu \ne \mu} p_\nu (v_\nu - v_\mu)^2 \le 2\,(1-p_\mu)\,\|v\|^2$, using $(v_\nu - v_\mu)^2 \le 2 v_\nu^2 + 2 v_\mu^2$; hence $\|H(p)\| \le 2(1-p_\mu)$ and $\beta\|X H X^{\top}\| \le 2\beta\sigma_{\max}^2(X)(1-p_\mu)$.
\end{proof}

The quantity $\ell_\mu(r)$ is exponentially small in $\beta(\Delta_\mu - 2Mr)$ and controls everything; it bounds the deviation of the Hessian from the identity, hence both the local strong convexity of $E$ and the Lipschitz constant of $x \mapsto Xp(x)$ on the ball. The standing assumption of the two theorems below is
\begin{equation}\label{eq:retrieval-cond}
2M\,\varepsilon_\mu(r) \;<\; r
\qquad\text{and}\qquad
\ell_\mu(r) \;<\; 1 ,
\end{equation}
a well-separation condition of the same nature as those of \citet{ramsauer2021} and \citet{hu2023sparse}, satisfied for fixed $r$ once $\beta \Delta_\mu$ is moderately large.

\begin{theorem}[Local contraction of the damped attention family]\label{thm:local-family}
Assume \eqref{eq:retrieval-cond} and let $\theta \in (0,1]$. Then \textup{(a)} $\Psi_\theta$ maps $\bar{B}(\xi_\mu, r)$ into itself; \textup{(b)} $\Psi_\theta$ is a contraction on $\bar{B}(\xi_\mu, r)$ with factor
\[
q_\theta \;=\; 1 - \theta\,\big(1 - \ell_\mu(r)\big) \;<\; 1 ,
\]
so there is a unique fixed point $x^\ast_\mu \in \bar{B}(\xi_\mu, r)$, which is a critical point of $E$ and satisfies $\|x^\ast_\mu - \xi_\mu\| \le 2M\varepsilon_\mu(r)$, and $\|x^{k} - x^\ast_\mu\| \le q_\theta^{\,k}\, \|x^{0} - x^\ast_\mu\|$ for every $x^0$ in the ball; \textup{(c)} $E$ is strongly convex on $\bar{B}(\xi_\mu, r)$ with modulus $1 - \ell_\mu(r)$, so $x^\ast_\mu$ is the unique local minimizer of $E$ there and is exponentially stable for the flow \eqref{eq:flow} with rate $1 - \ell_\mu(r)$.
\end{theorem}

\begin{proof}
For (a), $\|\Psi_\theta(x) - \xi_\mu\| \le (1-\theta)\|x - \xi_\mu\| + \theta\, \|Xp(x) - \xi_\mu\| \le (1-\theta)r + 2\theta M \varepsilon_\mu(r) \le r$ by Lemma~\ref{lem:concentration}(ii) and \eqref{eq:retrieval-cond}. For (b), $D\Psi_\theta(x) = (1-\theta) I + \theta\, \beta X H(p(x)) X^{\top}$ is symmetric positive semidefinite with $\|D\Psi_\theta(x)\| \le (1-\theta) + \theta\, \ell_\mu(r) = q_\theta$ on the ball by Lemma~\ref{lem:concentration}(iii), and the ball is convex, so the mean value inequality applies; Banach's theorem gives the unique fixed point, which satisfies $x^\ast_\mu = Xp(x^\ast_\mu)$ by Proposition~\ref{prop:family} and hence $\|x^\ast_\mu - \xi_\mu\| \le 2M\varepsilon_\mu(r)$ by Lemma~\ref{lem:concentration}(ii) again. Part (c) restates Lemma~\ref{lem:concentration}(iii) as $\nabla^2 E \succeq (1-\ell_\mu(r)) I$ on the ball.
\end{proof}

For $\theta = 1$ the factor degenerates to $q_1 = \ell_\mu(r)$, which is exponentially small: a single attention step contracts the ball around the attractor by an exponentially small factor, and lands within $2M\varepsilon_\mu(r)$ of the stored pattern from anywhere in the ball. This recovers, in fixed point language and with an explicit contraction rate, the one-step retrieval phenomenon of \citet{ramsauer2021}; the additional information carried by Theorem~\ref{thm:local-family} is the entire curve $\theta \mapsto q_\theta$, which shows that among underrelaxations $\theta \le 1$ the attention endpoint is locally optimal and that underrelaxation degrades the local rate linearly, in parallel with the global constant of Corollary~\ref{cor:rates}. The restriction to $\theta \le 1$ is, however, not forced by the local theory, and removing it changes the answer.

\begin{theorem}[Overrelaxed retrieval: the full local window]\label{thm:overrelax}
Assume \eqref{eq:retrieval-cond}, write $\ell = \ell_\mu(r)$ and $\rho = r - 2M\varepsilon_\mu(r) > 0$, and let $x^\ast_\mu$ be the attractor of Theorem~\ref{thm:local-family}. Then for every $\theta \in (0,2)$ the map $\Psi_\theta$ leaves $\bar{B}(x^\ast_\mu, \rho)$ invariant and contracts on it with certified factor
\[
q_\theta \;=\; \max\big\{\, |1-\theta|,\ |1 - \theta(1-\ell)| \,\big\} \;<\; 1,
\]
which for $\theta \in (0,1]$ reduces to the factor $1 - \theta(1-\ell)$ of Theorem~\ref{thm:local-family}. Over the window, $q_\theta$ is minimized at
\[
\theta_\star \;=\; \frac{2}{2-\ell}, \qquad q_{\theta_\star} \;=\; \frac{\ell}{2-\ell} \;<\; \ell \;=\; q_1 :
\]
a certified slight overrelaxation strictly improves on the attention update, by the factor $1/(2-\ell) \in (\tfrac12, 1)$. Every $\theta \in (0,2)$ lies simultaneously in the dissipation window of Theorem~\ref{thm:dissipation} and in the cell window of Theorem~\ref{thm:cells}, so overrelaxed retrieval retains all global guarantees.
\end{theorem}

The preservation statement is not confined to the quadratic modern-Hopfield geometry; Section~\ref{sec:generalized} returns to overrelaxation in dual coordinates, and Corollary~\ref{cor:cells-overrelax} shows that the same energy-cell mechanism survives for damped difference-of-convex iterations under bounded Bregman asymmetry and dual-domain admissibility.

\begin{proof}
$D\Psi_\theta(x) = (1-\theta)I + \theta\,\beta X H(p(x)) X^{\top}$ is symmetric with spectrum contained in $[\,1-\theta,\ 1-\theta+\theta\ell\,]$ on $\bar{B}(\xi_\mu, r)$ by Lemma~\ref{lem:concentration}\textup{(iii)}, so its norm is bounded there by $q_\theta$, and $q_\theta < 1$ precisely for $\theta \in (0,2)$, since $|1-\theta| < 1 \Leftrightarrow \theta \in (0,2)$ while $|1-\theta(1-\ell)| < 1 \Leftrightarrow \theta \in (0, 2/(1-\ell)) \supseteq (0,2)$. The ball $\bar{B}(x^\ast_\mu, \rho)$ is contained in $\bar{B}(\xi_\mu, r)$ because $\|x^\ast_\mu - \xi_\mu\| \le 2M\varepsilon_\mu(r)$, it is convex, and it contains the fixed point, so the mean value inequality along its segments gives $\|\Psi_\theta(x) - x^\ast_\mu\| \le q_\theta\,\|x - x^\ast_\mu\|$, which is simultaneously the contraction and the invariance. For the minimax, on $(0,2)$ the branch $|1-\theta(1-\ell)| = 1-\theta(1-\ell)$ is decreasing up to its zero at $1/(1-\ell)$ while $|1-\theta|$ increases past $\theta = 1$; the increasing branch $\theta - 1$ meets the decreasing branch $1 - \theta(1-\ell)$ at $\theta_\star = 2/(2-\ell) < 1/(1-\ell)$ with common value $\ell/(2-\ell)$, beyond which $\theta - 1$ dominates, so the maximum is minimized there.
\end{proof}

\begin{remark}[Kernel modes, the oracle cap, and certificate slack]\label{rem:oracle}
Two structural facts temper Theorem~\ref{thm:overrelax} in practice, and Campaign~G of Section~\ref{sec:experiments} quantifies both. First, $H(p)\mathbf{1} = 0$, so $A(x) = \beta X H(p(x)) X^{\top}$ has rank at most $N-1$; whenever $N \le d$ the smallest eigenvalue of $A(x^\ast_\mu)$ is exactly zero, the spectral radius of $D\Psi_\theta(x^\ast_\mu)$ for $\theta \ge 1$ equals $\max\{\theta-1,\ |1-\theta+\theta\lambda_{\max}|\}$ with $\lambda_{\max} = \lambda_{\max}(A(x^\ast_\mu))$, and the oracle optimum $\theta_{\mathrm{o}} = 2/(2-\lambda_{\max})$ improves the asymptotic factor to $\lambda_{\max}/(2-\lambda_{\max})$, at best a factor of two; measured in iterations to a fixed tolerance, the relative saving is $1 - k_{\mathrm{o}}/k_1 = \log(2-\lambda_{\max})\big/\log\big((2-\lambda_{\max})/\lambda_{\max}\big)$, since $k_1 \simeq \log(e_0/\varepsilon)/\log(1/\lambda_{\max})$ and $k_{\mathrm{o}} \simeq \log(e_0/\varepsilon)/\log((2-\lambda_{\max})/\lambda_{\max})$; this tends to zero as $\lambda_{\max} \to 0$ and is negligible in the well-separated regime. Second, the certified $\theta_\star$ is computed from $\ell$, which bounds the ball supremum of $\|A\|$ through the constants $2(N-1)$ of Lemma~\ref{lem:concentration}; the certified factor is attained exactly at $\theta_\star$, the kernel modes realizing $|1-\theta_\star| = q_{\theta_\star}$, but it improves on the \emph{observed} attention factor $\lambda_{\max}$ only if $\ell/(2-\ell) < \lambda_{\max}$, that is, only if the certificate is tight within a factor $2-\ell$ of the truth. In our experiments the slack $\ell/\lambda_{\max}$ never drops below $2(N-1)$, and plain attention beats the certified overrelaxation in observed iterations throughout the sweep; the practical value of $\theta_\star$ is therefore conditional on sharper certified bounds for $\sup_{\bar B}\|A\|$, which we leave open.
\end{remark}

\begin{theorem}[Unconditional local contraction of the local branch]\label{thm:local-ie}
Assume \eqref{eq:retrieval-cond} and let $\Delta t > 0$ be arbitrary. Then for every $x \in \bar{B}(\xi_\mu, r)$ the implicit equation \eqref{eq:ie} has a unique solution $\Phi_{\Delta t}(x)$ in $\bar{B}(\xi_\mu, r)$, the \emph{local branch} of Definition~\ref{def:ie-objects}, obtainable as the limit of the inner iteration $y_{j+1} = \big(x + \Delta t\, Xp(y_j)\big)/(1+\Delta t)$, which is itself a contraction of the ball with factor $\tfrac{\Delta t}{1+\Delta t}\,\ell_\mu(r) < 1$. The map $\Phi_{\Delta t}$ leaves the ball invariant, is a contraction with factor
\[
\kappa_{\Delta t} \;=\; \frac{1}{1 + \Delta t\,\big(1 - \ell_\mu(r)\big)} \;<\; 1 ,
\]
and its unique fixed point in the ball is the same attractor $x^\ast_\mu$ of Theorem~\textup{\ref{thm:local-family}}. When $\Delta t\,\nu < 1$, the global proximal map is a singleton and its unique selection coincides with the local branch; outside that regime the two objects may differ (Proposition~\ref{prop:tunneling}).
\end{theorem}

\begin{proof}
Write the implicit step as the fixed point equation $y = T_x(y) := \big(x + \Delta t\, Xp(y)\big)/(1+\Delta t)$. By Lemma~\ref{lem:concentration}(ii) and \eqref{eq:retrieval-cond}, $\|T_x(y) - \xi_\mu\| \le \big(\|x-\xi_\mu\| + \Delta t\, 2M\varepsilon_\mu(r)\big)/(1+\Delta t) \le r$ for $x, y$ in the ball, and the Lipschitz constant of $Xp(\cdot)$ on the ball is at most $\ell_\mu(r)$ by Lemma~\ref{lem:concentration}(iii), so $T_x$ contracts the ball with factor $\tfrac{\Delta t}{1+\Delta t}\ell_\mu(r) < 1$ and has a unique fixed point $\Phi_{\Delta t}(x)$ there. For the outer contraction, subtracting the fixed point equations for $x$ and $x'$ gives $\|\Phi_{\Delta t}(x) - \Phi_{\Delta t}(x')\| \le \tfrac{1}{1+\Delta t}\|x - x'\| + \tfrac{\Delta t}{1+\Delta t}\,\ell_\mu(r)\,\|\Phi_{\Delta t}(x) - \Phi_{\Delta t}(x')\|$, and solving for the left-hand side yields the factor $\kappa_{\Delta t}$. Fixed points of $\Phi_{\Delta t}$ solve $\nabla E = 0$ in the ball, and Theorem~\ref{thm:local-family} identifies the unique such point as $x^\ast_\mu$.
\end{proof}

\begin{remark}[The two implicit limits]\label{rem:two-limits}
Theorems~\ref{thm:local-family} and \ref{thm:local-ie} exhibit a structural asymmetry between the two ways of letting $\Delta t \to \infty$. The semi-implicit family tends to the attention update, one exact minimization of the convex majorizer, with local factor $\ell_\mu(r)$; the fully implicit scheme tends to the exact equilibrium computation $x = Xp(x)$, with local factor $\kappa_{\Delta t} \to 0$, that is, to the deep equilibrium formulation of the retrieval problem in the sense of \citet{bai2019}, the formulation whose accelerated solution for classical Hopfield energies is the subject of \citet{goemaere2024}. The comparison $\kappa_{\Delta t} < q_1 = \ell_\mu(r)$ holds precisely for $\Delta t > 1/\ell_\mu(r)$, so the implicit scheme outperforms attention per outer step only at very large step sizes; since each outer step is computed by the inner iteration of Theorem~\ref{thm:local-ie}, whose cost unit is the same softmax evaluation as one attention step, the per-evaluation accounting, carried out in Section~\ref{sec:cost}, shows that within the warm-started Picard hierarchy no allocation of inner iterations certifies a better per-evaluation factor than the attention bound $\ell$, and its interest lies instead in the regimes where \eqref{eq:retrieval-cond} fails, near basin boundaries and at moderate $\beta$, where dissipation without step size restriction is the property that survives.
\end{remark}

% --- end of Section 3 ---

% Section 4 -- draft v1
% Assumes: amsmath, amsthm, natbib; environments theorem, lemma, proposition, corollary, remark, definition numbered within section.
% Depends on: Section 2 (eq:energy, eq:family, eq:ie, lem:structure, lem:majorization, prop:family),
%             Section 3 (thm:dissipation, thm:single-limit, lem:concentration, eq:retrieval-cond, thm:local-family).

\section{Discrete versus continuous basins of attraction}\label{sec:basins}

The results of Section~\ref{sec:dissipation} guarantee that every orbit converges to some critical point of $E$; for a memory, the question that matters is to \emph{which} one. This section compares the basins of attraction of the discrete schemes with those of the flow \eqref{eq:flow}. The main positive result, Theorem~\ref{thm:cells}, states that below the escape energy the basins of the flow, of the whole damped attention family, and of the implicit Euler scheme in its uniqueness regime, all contain the same energy cells; at this level of description the discretization does not erode the basin at all, for any step size in the admissible ranges. Two negative results, Propositions~\ref{prop:tunneling} and \ref{prop:collapse}, exhibit the complementary mechanisms by which preservation can fail: the globally selected proximal step tunnels out of non-global cells beyond a quantified step size, and the explicit Euler method turns the attractor into a repeller beyond $\theta \approx 2$, at which point its basin loses interior altogether. Together with the sandwich estimate of Corollary~\ref{cor:sandwich}, which places each certified cell between two concentric balls of nearly equal radii in the well-separated regime, the section identifies a common certified core of the basins below the escape energy and localizes all possible discrepancies between discrete and continuous basins in the region above it; it does not bound the size of those discrepancies, a question left open in Remark~\ref{rem:frontier}.

\subsection{Energy cells and the monotone interpolation lemma}\label{sec:cells}

\begin{definition}[Basins and cells]\label{def:cells}
For the flow \eqref{eq:flow} and an equilibrium $x^\ast$, let $\mathcal{B}(x^\ast) = \{x^0 : x(t; x^0) \to x^\ast \text{ as } t \to \infty\}$ denote its basin of attraction, and for a one-step map $S$ with fixed point $x^\ast$ let $\mathcal{B}_S(x^\ast) = \{x^0 : S^k(x^0) \to x^\ast\}$. For $c > E(x^\ast)$, the \emph{energy cell} $U_c(x^\ast)$ is the connected component of the open sublevel set $\{E < c\}$ containing $x^\ast$; by coercivity (Lemma~\ref{lem:structure}(d)) every cell is bounded. The \emph{escape energy} of $x^\ast$ is
\[
c^\ast(x^\ast)
\;:=\;
\sup\Bigl(
\{E(x^\ast)\}
\cup
\bigl\{c>E(x^\ast): U_c(x^\ast)\text{ contains no critical point of }E\text{ other than }x^\ast\bigr\}
\Bigr).
\]
\end{definition}

Thus $c^\ast(x^\ast)=E(x^\ast)$ when no nontrivial admissible level exists, as may happen at a non-isolated minimizer, in which case the results below are vacuous. Since the cells are nested and increasing in $c$, the threshold is otherwise well defined; when $E$ is a Morse function, $c^\ast(x^\ast)$ is the first critical value at which the component of $x^\ast$ meets another critical point, generically the value of an adjacent index-one saddle. The endpoint and the energy floor inside an admissible cell are recorded explicitly next.

\begin{lemma}[Closure at the escape energy and the cell energy floor]\label{lem:closure}
If $E(x^\ast)<c^\ast(x^\ast)<\infty$, then $U_{c^\ast(x^\ast)}(x^\ast)$ contains no critical point of $E$ other than $x^\ast$. If $c^\ast(x^\ast)=E(x^\ast)$ there is no nontrivial admissible level. Consequently every finite level
\[
 E(x^\ast)<c\le c^\ast(x^\ast)
\]
is admissible in Theorem~\ref{thm:cells}; when $c^\ast(x^\ast)=\infty$, every finite $c>E(x^\ast)$ is admissible. Moreover, for every finite admissible level $c$,
\[
 E(y)\ge E(x^\ast)\qquad\text{for every }y\in U_c(x^\ast).
\]
\end{lemma}

\begin{proof}
Suppose first that $E(x^\ast)<c^\ast(x^\ast)<\infty$ and that $x^{\dagger}\ne x^\ast$ is critical with $x^{\dagger}\in U_{c^\ast(x^\ast)}(x^\ast)$. Then $E(x^{\dagger})<c^\ast(x^\ast)$ and there is a path $\gamma\subset\{E<c^\ast(x^\ast)\}$ joining $x^\ast$ to $x^{\dagger}$. Compactness of $\gamma$ and continuity of $E$ give $c':=\max_\gamma E<c^\ast(x^\ast)$. Choosing $\max\{c',E(x^{\dagger})\}<c''<c^\ast(x^\ast)$ gives $\gamma\subset\{E<c''\}$ and hence $x^{\dagger}\in U_{c''}(x^\ast)$, contradicting the definition of $c^\ast(x^\ast)$.

For the energy floor, let $c$ be any finite admissible level. Coercivity implies that $\overline{U_c(x^\ast)}$ is compact. Also
\[
 \partial U_c(x^\ast)\subseteq\{E=c\}.
\]
Indeed, if $z\in\partial U_c(x^\ast)$ had $E(z)<c$, then a sufficiently small ball about $z$ would lie in the open set $\{E<c\}$ and meet $U_c(x^\ast)$; their union would be connected, forcing $z\in U_c(x^\ast)$, a contradiction. Hence $E$ attains its minimum on $\overline{U_c(x^\ast)}$, and because $E(x^\ast)<c$ this minimizer cannot lie on the boundary. It is therefore an interior critical point of the cell. By admissibility the only such point is $x^\ast$, proving $E\ge E(x^\ast)$ throughout $U_c(x^\ast)$.
\end{proof}

The mechanism behind the preservation theorem is that the discrete orbits of the family admit a continuous interpolation along which the energy never rises. This is where the unit-curvature majorization of Lemma~\ref{lem:majorization} enters a second time, now evaluated along the whole segment between consecutive iterates rather than only at the endpoint.

\begin{lemma}[Monotone interpolation]\label{lem:interpolation}
Let $\theta \in (0,2]$, $x \in \mathbb{R}^d$, and let $y_t = (1-t)\,x + t\,\Psi_\theta(x)$ for $t \in [0,1]$ be the segment joining $x$ to its update. Then
\[
E(y_t) \;\le\; E(x) \;-\; t\theta\,\Big(1 - \tfrac{t\theta}{2}\Big)\,\big\|\nabla E(x)\big\|^2 \;\le\; E(x)
\qquad\text{for all } t \in [0,1].
\]
\end{lemma}

\begin{proof}
Since $y_t = x - t\theta\,\nabla E(x)$, the majorization \eqref{eq:majorization} gives $E(y_t) \le E(x) - t\theta\,\|\nabla E(x)\|^2 + \tfrac12 t^2\theta^2\,\|\nabla E(x)\|^2$, and $t\theta \in [0,2]$ makes the bracket nonnegative.
\end{proof}

We emphasize what makes Lemma~\ref{lem:interpolation} nontrivial in the present setting. For an $L$-smooth energy the same interpolation argument is folklore for steps below $2/L$; the content here is that, by Lemma~\ref{lem:majorization}, the energy \eqref{eq:energy} is majorized with constant exactly one uniformly in $\beta$, $N$ and the pattern geometry, even though its Hessian is unbounded below as $\beta$ grows, so the admissible range $\theta \in (0,2)$ is uniform over the entire multistable regime and, in particular, contains the attention endpoint $\theta = 1$ with a factor-two margin. For the implicit scheme the segment is the wrong interpolant, since the proximal step is not a gradient step from $x$; the correct one is the proximal trajectory.

\begin{lemma}[Proximal interpolation]\label{lem:prox-path}
Let $\Delta t\,\nu < 1$, with $\nu$ the one-sided Lipschitz constant of Lemma~\textup{\ref{lem:structure}(c)}, and for $x \in \mathbb{R}^d$ let $p_s(x)$ denote the unique proximal point $\operatorname{prox}_{sE}(x)$ for $s \in (0,\Delta t]$, with $p_0(x) = x$. Then $s \mapsto p_s(x)$ is continuous on $[0,\Delta t]$, $p_{\Delta t}(x)$ is the implicit Euler update, and $s \mapsto E(p_s(x))$ is nonincreasing.
\end{lemma}

\begin{proof}
For $s\,\nu < 1$ the proximal functional is $(\tfrac1s - \nu)$-strongly convex, so $p_s(x)$ is single valued and solves $p_s = x - s\,\nabla E(p_s)$; continuity in $s$, including $p_s \to x$ as $s \to 0^{+}$, follows from the implicit function theorem, the Jacobian $I + s\,\nabla^2E(p_s)$ being invertible in this regime. For the monotonicity, let $0 < s_1 < s_2 \le \Delta t$ and abbreviate $p_i = p_{s_i}(x)$. Optimality of $p_1$ at parameter $s_1$ against $p_2$, and of $p_2$ at parameter $s_2$ against $p_1$, give $E(p_1) + \tfrac{1}{2s_1}\|p_1 - x\|^2 \le E(p_2) + \tfrac{1}{2s_1}\|p_2 - x\|^2$ and $E(p_2) + \tfrac{1}{2s_2}\|p_2 - x\|^2 \le E(p_1) + \tfrac{1}{2s_2}\|p_1 - x\|^2$; adding them yields $\|p_1 - x\| \le \|p_2 - x\|$, and substituting back into the second inequality yields $E(p_2) \le E(p_1)$.
\end{proof}

\subsection{Preservation of energy cells}\label{sec:preservation}

\begin{theorem}[Cell preservation]\label{thm:cells}
Let $x^\ast$ be a local minimizer of $E$ and let $c$ be finite with $E(x^\ast) < c \le c^\ast(x^\ast)$, so that by Lemma~\ref{lem:closure} the cell $U_c(x^\ast)$ contains no critical point other than $x^\ast$. Then
\[
U_c(x^\ast) \;\subseteq\; \mathcal{B}(x^\ast) \;\cap\; \bigcap_{\theta \in (0,2)} \mathcal{B}_{\Psi_\theta}(x^\ast) \;\cap\; \bigcap_{0 < \Delta t < 1/\nu^{+}} \mathcal{B}_{\Phi_{\Delta t}}(x^\ast),
\]
where $\nu^{+} = \max\{\nu, 0\}$ and $1/\nu^{+} = \infty$ when $\nu \le 0$. In words, every energy cell below the escape energy lies in the basin of the flow, of every member of the damped attention family, and of the implicit Euler scheme throughout its uniqueness regime, simultaneously and with no further condition on the data.
\end{theorem}

\begin{proof}
We give the argument for $\Psi_\theta$; the flow and the implicit scheme are identical after replacing the segment interpolant by the trajectory itself and by the proximal path of Lemma~\ref{lem:prox-path}, respectively. Let $x^0 \in U_c(x^\ast)$ and let $\gamma$ be the piecewise-linear curve obtained by concatenating the segments $[x^{k}, x^{k+1}]$. By Lemma~\ref{lem:interpolation}, $E \le E(x^{k}) \le E(x^0) < c$ along the $k$-th segment, so $\gamma$ is a connected subset of $\{E < c\}$ containing $x^0$, hence $\gamma \subseteq U_c(x^\ast)$ and in particular the whole orbit remains in the cell, which is bounded, so the orbit is precompact. By Theorem~\ref{thm:dissipation}, $\nabla E(x^{k}) \to 0$ along the orbit, via \eqref{eq:descent-family} for the family and via \eqref{eq:descent-ie} together with the optimality condition $\nabla E(x^{k+1}) = -(x^{k+1}-x^{k})/\Delta t$ for the proximal selection. Every accumulation point $x^{\dagger}$ is therefore critical, satisfies $E(x^{\dagger}) \le E(x^{0}) < c$, and, being a limit of points of the component $U_c(x^\ast)$ with a ball neighborhood inside the open set $\{E < c\}$ meeting $U_c(x^\ast)$, lies in $U_c(x^\ast)$; the only critical point there is $x^\ast$, so $x^\ast$ is the unique accumulation point of a precompact orbit and $x^{k} \to x^\ast$. For the flow the same argument applies with the trajectory as interpolant, LaSalle's invariance principle supplying criticality of the accumulation points. No {\L}ojasiewicz-type argument is needed inside a cell.
\end{proof}

Theorem~\ref{thm:cells} is the precise sense in which the family $\Psi_\theta$ and the moderately stepped implicit scheme are structure preserving for retrieval, as not only the equilibria (Proposition~\ref{prop:family}) but the entire certified core of each basin passes to the discretization unchanged, uniformly in the step size parameter. Two readings are worth separating. Quantitatively, the theorem converts every lower bound on the escape energy into a common certified basin for all schemes at once; the next corollary provides such a bound, with explicit constants, in the well-separated regime; we reserve \emph{escape energy} for the topological quantity $c^\ast$, computable (Campaign~E) but not closed-form, and call the bound $c_r$ below the \emph{explicit certified escape level}. The corollary also shows that the resulting cells are sandwiched between two balls of comparable radii, sharpening the ball-shaped certified regions of Theorem~\ref{thm:local-family} in the same way that the relative nonlinear measure estimates of \citet{villatoro2005} sharpen pointwise contraction to a basin statement. Qualitatively, the theorem localizes all possible basin discrepancies between the discrete schemes and the flow in the region above the escape energy, near the stable manifolds of the saddles, which is where the numerical exploration of Section~\ref{sec:experiments} concentrates.

\begin{corollary}[Sandwich estimate for cells]\label{cor:sandwich}
Assume the retrieval condition \eqref{eq:retrieval-cond} on the ball $\bar{B}(\xi_\mu, r)$, write $\ell = \ell_\mu(r)$, and let $x^\ast_\mu$ be the attractor of Theorem~\textup{\ref{thm:local-family}}. Set
\[
c_r \;=\; E(x^\ast_\mu) \;+\; \tfrac12\,(1-\ell)\,\big(r - 2M\varepsilon_\mu(r)\big)^2 .
\]
Then for every $c$ with $E(x^\ast_\mu) < c \le c_r$: the cell $U_c(x^\ast_\mu)$ is contained in $\bar{B}(\xi_\mu, r)$, contains no critical point other than $x^\ast_\mu$, so that $c_r \le c^\ast(x^\ast_\mu)$ and Theorem~\textup{\ref{thm:cells}} applies to it, and moreover
\[
B\big(x^\ast_\mu,\ \rho_-\big) \;\subseteq\; U_c(x^\ast_\mu) \;\subseteq\; B\big(x^\ast_\mu,\ \rho_+\big),
\qquad
\rho_- = \sqrt{2\,\big(c - E(x^\ast_\mu)\big)}, \quad \rho_+ = \frac{\rho_-}{\sqrt{1-\ell}} \;>\; \rho_- .
\]
\end{corollary}

\begin{proof}
The global majorization \eqref{eq:majorization} at the critical point $x^\ast_\mu$ reads $E(y) \le E(x^\ast_\mu) + \tfrac12\|y - x^\ast_\mu\|^2$ for all $y$, so the open ball of radius $\rho_-$ lies in $\{E < c\}$; being connected and containing $x^\ast_\mu$, it lies in the cell, which proves the inner inclusion. For the outer inclusion, strong convexity on the ball, Theorem~\ref{thm:local-family}(c), gives $E(y) \ge E(x^\ast_\mu) + \tfrac12(1-\ell)\,\|y - x^\ast_\mu\|^2$ for $y \in \bar{B}(\xi_\mu, r)$, so every point of the cell that lies in $\bar{B}(\xi_\mu, r)$ satisfies $\|y - x^\ast_\mu\| < \rho_+$, and $\rho_+ \le r - 2M\varepsilon_\mu(r) \le r - \|x^\ast_\mu - \xi_\mu\|$ by the definition of $c_r$ and Theorem~\ref{thm:local-family}(b), so such points are interior to $\bar{B}(\xi_\mu, r)$ at distance at least $r - \|x^\ast_\mu - \xi_\mu\| - \rho_+ \ge 0$ from its boundary; hence the intersection of the cell with the closed ball is open and closed in the cell, and by connectedness the cell lies entirely in the ball. Uniqueness of the critical point in the ball is Theorem~\ref{thm:local-family}(b), completing the proof.
\end{proof}

Since $\ell$ is exponentially small in $\beta(\Delta_\mu - 2Mr)$, the two radii differ by a factor $1/\sqrt{1-\ell} = 1 + O(\ell)$; in the well-separated regime each certified cell is sandwiched between two concentric balls, centered at the attractor rather than at the pattern, whose radii agree to within an exponentially small relative factor, and Theorem~\ref{thm:cells} hands the whole certified ball to every scheme in the family at once. It is only above $c_r$, and ultimately above $c^\ast(x^\ast_\mu)$, that the geometry of the basins can become scheme dependent.

\subsection{The two failure modes}\label{sec:failure}

The restriction to the uniqueness regime $\Delta t\,\nu < 1$ in Theorem~\ref{thm:cells} is not an artifact of the proof. Beyond it, the globally selected proximal step of Theorem~\ref{thm:dissipation}(b) can leave the cell in a single iteration, and does so systematically once the step size exceeds an explicit threshold determined by the energy drop available elsewhere.

\begin{proposition}[Tunneling of the global proximal step]\label{prop:tunneling}
Let $x^\ast$ be a local minimizer that is not a global minimizer, let $x_g$ be a global minimizer, and let $\delta=E(x^\ast)-E(x_g)>0$. Fix $E(x^\ast)<c\le c^\ast(x^\ast)$ and $x^0\in U_c(x^\ast)$. If
\[
 \Delta t>\Delta t^\ast(x^0):=\frac{\|x_g-x^0\|^2}{2\delta},
\]
then every $x^1\in\mathcal{P}_{\Delta t}(x^0)$ satisfies $E(x^1)<E(x^\ast)$. Consequently $x^1\notin U_c(x^\ast)$, no subsequent orbit generated by global proximal minimizers $x^{k+1}\in\mathcal{P}_{\Delta t}(x^k)$ can enter $U_c(x^\ast)$, and every such orbit converges to a critical point different from $x^\ast$. In particular, for every single-valued global proximal selection $P_{\Delta t}$,
\[
 x^0\in\mathcal{B}(x^\ast)\setminus\mathcal{B}_{P_{\Delta t}}(x^\ast).
\]
\end{proposition}

\begin{proof}
Testing the proximal minimality at $x^1$ against the competitor $x_g$ gives $E(x^1) \le E(x^1) + \tfrac{1}{2\Delta t}\|x^1 - x^0\|^2 \le E(x_g) + \tfrac{1}{2\Delta t}\|x_g - x^0\|^2 < E(x_g) + \delta = E(x^\ast)$. Since $E \ge E(x^\ast)$ on $U_c(x^\ast)$ by Lemma~\ref{lem:closure}, $x^1$ lies outside the cell, and by the monotonicity of the energy along the orbit (Theorem~\ref{thm:dissipation}(b)) so do all later iterates; Theorem~\ref{thm:single-limit} then delivers convergence to a critical point of energy below $E(x^\ast)$, which cannot be $x^\ast$.
\end{proof}

Proposition~\ref{prop:tunneling} assigns complementary roles to the two implicit objects of Definition~\ref{def:ie-objects}, and applied at $x^0 = x^\ast$ itself it shows that a non-global local minimizer ceases to satisfy $x^\ast\in\mathcal{P}_{\Delta t}(x^\ast)$ for $\Delta t > \|x_g - x^\ast\|^2/(2\delta)$, the fact recorded after Proposition~\ref{prop:family}. As a global optimizer of \eqref{eq:energy} the large-step proximal scheme is excellent, reaching low energy in one step; as a memory it is defective, because retrieval is precisely the task of converging to the attractor designated by the initial condition, and tunneling answers a different question. The threshold $\Delta t^\ast(x^0)$ is explicit and finite on every cell of every non-global attractor, while cell preservation is guaranteed for $\Delta t\,\nu < 1$; the behavior in the intermediate window $[1/\nu, \Delta t^\ast]$, where the proximal map may be multivalued but need not yet tunnel, is not resolved by these results and is examined numerically in Section~\ref{sec:experiments}.

The second failure mode concerns the only member of the family able to leave $\theta \in (0,1]$, the explicit Euler method, and shows that the endpoint $\theta = 2$ of Lemma~\ref{lem:interpolation} is essentially sharp, since shortly beyond it the attractor does not merely shed basin volume but becomes a repeller.

\begin{proposition}[Overshoot collapse of explicit Euler]\label{prop:collapse}
Assume the retrieval condition \eqref{eq:retrieval-cond}, write $\ell = \ell_\mu(r)$, and assume that $E$ has at least one critical point other than $x^\ast_\mu$, as is the case whenever a second pattern satisfies \eqref{eq:retrieval-cond}, and as is guaranteed under the geometric conditions of \citet{beise2026}. If $\theta > 2/(1-\ell)$, then every eigenvalue of $D\Psi_\theta(x^\ast_\mu)$ has modulus at least $\theta(1-\ell) - 1 > 1$, so $x^\ast_\mu$ is a repelling fixed point of $\Psi_\theta$. Consequently there is a neighborhood $V$ of $x^\ast_\mu$ such that every orbit converging to $x^\ast_\mu$ reaches it exactly in finitely many steps, and the basin $\mathcal{B}_{\Psi_\theta}(x^\ast_\mu)$ has empty interior.
\end{proposition}

\begin{proof}
On the retrieval ball, $D\Psi_\theta(x) = (1-\theta)I + \theta\,\beta X H(p(x)) X^{\top}$ is symmetric with spectrum contained in $[\,1-\theta,\ 1-\theta+\theta\ell\,]$ by Lemma~\ref{lem:concentration}(iii); for $\theta > 2/(1-\ell)$ this interval lies in $(-\infty, -1)$, and all eigenvalues at $x^\ast_\mu$ have modulus at least $\theta - 1 - \theta\ell = \theta(1-\ell) - 1 > 1$. By continuity there are a neighborhood $V$ and $\lambda > 1$ with $\|\Psi_\theta(x) - x^\ast_\mu\| \ge \lambda\,\|x - x^\ast_\mu\|$ for $x \in V$, so orbits in $V \setminus \{x^\ast_\mu\}$ leave $V$, and an orbit can converge to $x^\ast_\mu$ only by landing on it exactly. If the set of initial points doing so contained an open set, then by the Baire category theorem some iterate $\Psi_\theta^{k}$ would be constant, equal to $x^\ast_\mu$, on a nonempty open set; $\Psi_\theta^{k}$ is real analytic, hence would be constant on $\mathbb{R}^d$, contradicting the assumed existence of a second critical point, which is a fixed point of $\Psi_\theta^{k}$.
\end{proof}

The threshold has a classical reading. If the flow carries a relaxation time $\tau$, $\dot{x} = \tau^{-1}(-x + Xp(x))$, then the unit-step explicit Euler convention of the neural network literature corresponds to $\theta = 1/\tau$, and Proposition~\ref{prop:collapse} predicts collapse for $\tau < \tfrac12(1-\ell) \approx \tfrac12$: the instability of unit-step synchronous discretizations, documented since the Takeda--Goodman model \citep{takeda1986} and part of the original motivation for the implicit schemes of \citet{atencia2002} and \citet{villatoro2005}, reappears here in sharp quantitative form, and the safety of the attention update is re-derived as the statement $\theta = 1 < 2$, with margin uniform in $\beta$.

\begin{remark}[What is not preserved, and what breaks at higher order]\label{rem:frontier}
Above the escape energy, Theorem~\ref{thm:cells} is silent, and genuinely so; there the continuous basin is bounded by the stable manifolds of saddle points, whose systematic existence, attached to higher-dimensional faces of the pattern polytope, has recently been established by \citet{beise2026}, and which discretization does deform; the discrete basin boundaries of $\Psi_\theta$ may in principle acquire the complicated structure familiar from discretized gradient systems. Quantifying this deformation, for instance through the measure of the symmetric difference of basins as a function of $\theta$ and $\beta$, appears to require tools beyond energy monotonicity, and is investigated numerically in Section~\ref{sec:experiments}. Separately, the two lemmas that power this section fail for generic second-order integrators, whose stages are not gradient steps of $E$ and do not in general furnish a monotone interpolant; restoring a discrete energy law at order two, via second-order convex splitting or scalar auxiliary variable reformulations, is the subject of Section~\ref{sec:order}, and Theorem~\ref{thm:cells} extends verbatim to any scheme for which a monotone interpolant can be exhibited, which is the property we shall demand of them.
\end{remark}

% --- end of Section 4 ---

% Section 5 -- draft v1
% Assumes: amsmath, amsthm, natbib; environments theorem, lemma, proposition, corollary, remark, definition numbered within section.
% Depends on: Section 2 (eq:energy, eq:family, prop:family, lem:structure, lem:majorization),
%             Section 3 (lem:concentration, eq:retrieval-cond), Section 4 (thm:cells, rem:frontier).

\section{Temporal accuracy and second-order energy-stable schemes}\label{sec:order}

Proposition~\ref{prop:family} showed that the convex splitting and exponential schemes generate identical orbits, so that any distinction between them, and more generally any ranking of the family members finer than the parameter $\theta$, must come from trajectory accuracy rather than from the equilibrium reached. This section quantifies that accuracy. Section~\ref{sec:constants} computes the leading error constants of the family and shows that the exponential member is distinguished after all; its error constant is proportional to the nonlinear curvature $\beta X H(p) X^{\top}$ and is therefore exponentially small on retrieval balls. Section~\ref{sec:barrier} shows that no reparametrization of the family can reach second order, so higher accuracy requires leaving the family. Section~\ref{sec:sav} constructs a second-order scheme that does so while retaining an unconditional discrete energy law, a scalar auxiliary variable (SAV) discretization in the sense of \citet{shen2018sav} (see also the relaxed variants of \citealp{jiang2022rsav,relaxedsav2023} and the original-energy-stable Lagrange multiplier and improved SAV schemes discussed after Theorem~\ref{thm:sav}), whose per-step cost is a single softmax evaluation and whose well-posedness is uniform in $\beta$.

\subsection{Error constants within the family}\label{sec:constants}

Write $f(x) = -\nabla E(x) = -x + Xp(x)$ for the vector field of the flow \eqref{eq:flow}, and recall $Df(x) = -\nabla^2E(x) = -I + \beta X H(p(x)) X^{\top}$.

\begin{proposition}[Consistency and leading error]\label{prop:lte}
Each member of Definition~\textup{\ref{def:schemes}} with parameter map $\theta(\Delta t) \in \{\Delta t,\ \Delta t/(1+\Delta t),\ 1-e^{-\Delta t}\}$ is consistent of order one, with local truncation error
\[
x(t+\Delta t) - \Psi_{\theta(\Delta t)}\big(x(t)\big) \;=\; \Delta t^{2}\, C(x(t)) \;+\; O(\Delta t^{3}),
\]
where the leading error constant is
\[
C_{\mathrm{EE}}(x) = \tfrac12\, Df(x)f(x), \qquad
C_{\mathrm{CS}}(x) = f(x) + \tfrac12\, Df(x)f(x), 
\]
\[
C_{\mathrm{ETD}}(x) = \tfrac12\,\big(I + Df(x)\big)f(x) = \tfrac12\,\beta X H(p(x)) X^{\top} f(x).
\]
In particular, writing $Df = -I + A$ with $\|A\| \le \ell_\mu(r)$ on $\bar{B}(\xi_\mu,r)$ by Lemma~\ref{lem:concentration}\textup{(iii)}, one has $C_{\mathrm{EE}} = \tfrac12(-I+A)f$ and $C_{\mathrm{CS}} = \tfrac12(I+A)f$, hence the two-sided window $\tfrac{1-\ell_\mu(r)}{2}\|f\| \le \|C_{\mathrm{EE}}\|,\ \|C_{\mathrm{CS}}\| \le \tfrac{1+\ell_\mu(r)}{2}\|f\|$, while $\|C_{\mathrm{ETD}}\| \le \tfrac{\ell_\mu(r)}{2}\|f\|$.
\end{proposition}

\begin{proof}
Taylor expansion of the exact solution gives $x(t+\Delta t) = x + \Delta t f + \tfrac12\Delta t^2 Df\,f + O(\Delta t^3)$, while the scheme produces $x + \theta(\Delta t) f$. Expanding $\theta_{\mathrm{EE}} = \Delta t$, $\theta_{\mathrm{CS}} = \Delta t - \Delta t^2 + O(\Delta t^3)$ and $\theta_{\mathrm{ETD}} = \Delta t - \tfrac12\Delta t^2 + O(\Delta t^3)$ and subtracting yields the three constants. The bound for $C_{\mathrm{ETD}}$ is Lemma~\ref{lem:concentration}(iii), since $I + Df = \beta X H(p) X^{\top}$.
\end{proof}

The computation resolves the degeneracy left by Proposition~\ref{prop:family} in the only way it could be resolved. As maps, CS and ETD coincide up to relabeling of $\Delta t$; as integrators, ETD assigns to each point of the common orbit a time label that is exact on the linear part of the flow, and the residual error is carried entirely by the softmax curvature, which the concentration Lemma~\ref{lem:concentration} makes exponentially small near well-separated patterns. On retrieval balls the exponential scheme therefore tracks the continuous trajectory with an error constant smaller by the guaranteed factor $\ell_\mu(r)/(1-\ell_\mu(r)) = \ell_\mu(r) + O(\ell_\mu(r)^2)$ than either Euler variant, and in the single-pattern case $N = 1$, where $f$ is affine, it is exact. Whenever the trajectory itself carries meaning, as in the interpretation of intermediate transformer layers as intermediate times of the flow \citep{hoover2023}, or when the discrete dynamics is used to estimate continuous escape times in Section~\ref{sec:experiments}, ETD is the preferred member of the family; for pure retrieval the distinction is invisible, consistently with Section~\ref{sec:cost}.

\subsection{An order barrier for the family}\label{sec:barrier}

\begin{proposition}[Order barrier]\label{prop:barrier}
Suppose there exists a point $x_0$ with $Df(x_0)f(x_0) \notin \operatorname{span}\{f(x_0)\}$. Then no parameter map $\theta(\Delta t) = \Delta t + O(\Delta t^2)$ makes the scheme $x^{k+1} = \Psi_{\theta(\Delta t)}(x^{k})$ consistent of order two for the flow \eqref{eq:flow}; second order requires leaving the family, through additional stages, additional steps, or implicitness in the nonlinear part. The hypothesis holds for $N = 2$ whenever the patterns are linearly independent (Lemma~\ref{lem:genericity}); for $N \ge 3$ we expect it to hold generically but do not prove it.
\end{proposition}

\begin{proof}
The order-two condition requires the $\Delta t^2$ term of the scheme, which by the expansion above lies in $\operatorname{span}\{f(x)\}$ for every admissible $\theta(\cdot)$, to equal $\tfrac12 Df(x) f(x)$ for all $x$; at $x_0$ this is impossible.
\end{proof}

\begin{lemma}[Genericity for two patterns]\label{lem:genericity}
Let $N = 2$ and let $\xi_1, \xi_2$ be linearly independent. Then there exists $x_0$ with $Df(x_0)f(x_0) \notin \operatorname{span}\{f(x_0)\}$. The proof is given in Appendix~\ref{app:barrier}.
\end{lemma}

The barrier is elementary but consequential when combined with Remark~\ref{rem:frontier}: the schemes that certify basins in Section~\ref{sec:basins} are first order, and generic second-order integrators, such as Heun or the implicit midpoint rule applied to \eqref{eq:flow}, do not in general furnish the monotone interpolant or the unconditional true-energy law that power Theorem~\ref{thm:cells}. The scheme of the next subsection is designed to give up neither.

\subsection{A second-order SAV discretization}\label{sec:sav}

Rather than treating the full log-sum-exp as the SAV nonlinearity, which is unbounded below and would force an artificial truncation outside an absorbing set, we resplit the energy. For a fixed $\gamma \in (0,1)$, write
\[
E(x) \;=\; \frac{1-\gamma}{2}\,\|x\|^2 \;+\; E_1^{\gamma}(x),
\qquad
E_1^{\gamma}(x) \;=\; \frac{\gamma}{2}\,\|x\|^2 \;-\; \operatorname{lse}_\beta\big(X^{\top}x\big).
\]
Since $\operatorname{lse}_\beta(X^{\top}x) \le M\|x\| + \beta^{-1}\log N$, completing the square gives $E_1^{\gamma}(x) + C_0 \ge \big(\sqrt{\gamma/2}\,\|x\| - M/\sqrt{2\gamma}\,\big)^2 + 1 \ge 1$ on all of $\mathbb{R}^d$ for the explicit constant $C_0 = M^2/(2\gamma) + \beta^{-1}\log N + 1$: the auxiliary variable $r \approx \sqrt{E_1^{\gamma}(x)+C_0}$ and the vector $b(x) = \nabla E_1^{\gamma}(x)/\sqrt{E_1^{\gamma}(x)+C_0}$, with $\nabla E_1^{\gamma}(x) = \gamma x - Xp(x)$, are globally defined with no truncation. Moreover $\|b(x)\| \to \sqrt{2\gamma}$ as $\|x\| \to \infty$, so $b$ is globally bounded and Lipschitz, the fact on which the convergence proof of Appendix~\ref{app:sav} rests, and the single expensive ingredient of $b$ remains one softmax evaluation. We take $\gamma = \tfrac12$ in the experiments.

\begin{definition}[SAV--CN scheme]\label{def:sav}
Given $x^{k-1}, x^{k}, r^{k}$, set $\bar{x}^{k+\frac12} = \tfrac32 x^{k} - \tfrac12 x^{k-1}$ and $b^{k} = b(\bar{x}^{k+\frac12})$, and define $(x^{k+1}, r^{k+1})$ by
\begin{equation}\label{eq:sav}
\frac{x^{k+1}-x^{k}}{\Delta t} \;=\; -\,(1-\gamma)\,\frac{x^{k+1}+x^{k}}{2} \;-\; b^{k}\,\frac{r^{k+1}+r^{k}}{2},
\qquad\qquad
r^{k+1}-r^{k} \;=\; \tfrac12\,\big\langle b^{k},\, x^{k+1}-x^{k}\big\rangle,
\end{equation}
initialized with $r^{0} = \sqrt{E_1^{\gamma}(x^{0})+C_0}$, one startup step of the ETD scheme for $x^{1}$, and $r^{1} = \sqrt{E_1^{\gamma}(x^{1})+C_0}$, which preserves the second-order accuracy of the startup.
\end{definition}

\begin{theorem}[Unconditional stability, solvability, and second-order convergence of SAV--CN]\label{thm:sav}
For every $\Delta t > 0$ and every SAV step $k \ge 1$, the system \eqref{eq:sav} has a unique solution, computable in closed form at the cost of one softmax evaluation plus $O(d)$ operations, by eliminating $r^{k+1}$ and solving the rank-one perturbed system $\big[(\tfrac{1}{\Delta t} + \tfrac{1-\gamma}{2}) I + \tfrac14\, b^{k} (b^{k})^{\top}\big] x^{k+1} = g^{k}$ with the Sherman--Morrison formula, where $g^{k} = \big[(\tfrac{1}{\Delta t} - \tfrac{1-\gamma}{2}) I + \tfrac14\, b^{k} (b^{k})^{\top}\big] x^{k} - b^{k} r^{k}$ and the matrix is symmetric positive definite. The modified energy $\tilde{E}^{k} = \tfrac{1-\gamma}{2}\|x^{k}\|^2 + (r^{k})^2 - C_0$ obeys the exact discrete law
\begin{equation}\label{eq:sav-law}
\tilde{E}^{k+1} \;=\; \tilde{E}^{k} \;-\; \frac{1}{\Delta t}\,\big\|x^{k+1}-x^{k}\big\|^{2}
\qquad\text{for every } \Delta t > 0,
\end{equation}
which in particular yields the unconditional a priori bound $\tfrac{1-\gamma}{2}\|x^{k}\|^2 \le \tilde{E}^{1} + C_0$ for all $k\ge1$. Solvability and the energy law hold for every step size. Moreover, for every fixed $T>0$, as $\Delta t\to0$ the numerical solution of the extended system satisfies $\max_{0\le k\le T/\Delta t}(\|x^k-x(t_k)\|+|r^k-r(t_k)|)\le C_T\Delta t^2$, with $C_T$ independent of $\Delta t$; see Appendix~\ref{app:sav}.
\end{theorem}

\begin{proof}
Taking the inner product of the first equation of \eqref{eq:sav} with $x^{k+1}-x^{k}$ and using the identities $\langle \tfrac{1-\gamma}{2}(x^{k+1}+x^{k}), x^{k+1}-x^{k}\rangle = \tfrac{1-\gamma}{2}(\|x^{k+1}\|^2 - \|x^{k}\|^2)$ and, by the second equation, $\langle b^{k}\tfrac12(r^{k+1}+r^{k}), x^{k+1}-x^{k}\rangle = (r^{k+1})^2 - (r^{k})^2$, gives \eqref{eq:sav-law} after rearrangement; the a priori bound follows since $(r^{k})^2 \ge 0$ and $\tilde{E}^{k}$ is nonincreasing. For solvability, substituting $r^{k+1} = r^{k} + \tfrac12\langle b^{k}, x^{k+1}-x^{k}\rangle$ into the first equation produces the stated linear system with a symmetric positive definite matrix, invertible in $O(d)$ by Sherman--Morrison. Consistency and convergence are proved in Appendix~\ref{app:sav}, using the extended system $\dot{x} = -(1-\gamma)x - b(x)\,r$, $\dot{r} = \tfrac12\langle b(x), \dot{x}\rangle$, which along solutions initialized with $r(0) = \sqrt{E_1^{\gamma}(x(0))+C_0}$ satisfies $r(t) = \sqrt{E_1^{\gamma}(x(t))+C_0}$ identically, so that $b(x)\,r = \nabla E_1^{\gamma}(x)$ and the system reduces to the flow \eqref{eq:flow}.
\end{proof}

Three remarks delimit what Theorem~\ref{thm:sav} does and does not achieve. First, the dissipated quantity is the modified energy $\tilde{E}^{k}$, in which $(r^{k})^2$ tracks $E_1^{\gamma}(x^{k})+C_0$ only up to a consistency error of order $\Delta t^2$; the true energy $E(x^{k})$ need not decrease, and Campaign~F of Section~\ref{sec:experiments} shows that at large steps it genuinely does not. Second-order schemes that dissipate the true energy of a gradient flow exactly are classical; the discrete-gradient methods of \citet{gonzalez1996} and \citet{mqr1999} and the average vector field method \citep{celledoni2012} satisfy $E(x^{k+1}) - E(x^{k}) = -\tfrac{1}{\Delta t}\|x^{k+1}-x^{k}\|^2$ by construction. For the energy \eqref{eq:energy} the midpoint discrete gradient of \citet{gonzalez1996} is implementable at the cost of one softmax-gradient evaluation plus two energy evaluations per step, each energy costing a single product with $X^{\top}$, and it is the natural exact-dissipation benchmark for SAV; it is, however, fully implicit, and whether its steps admit a monotone interpolant in the sense of Section~\ref{sec:basins}, so that Theorem~\ref{thm:cells} extends to it, is not known to us. Within the SAV family itself, relaxation techniques \citep{jiang2022rsav,relaxedsav2023} bring the modified energy closer to the true one, and Lagrange-multiplier and improved SAV variants restore stability of the \emph{original} energy while remaining linear, at the price, in the Lagrange-multiplier case, of a step-size restriction for solvability; we have not tested them here. The open problem left by this section is therefore not existence but combination; a scheme achieving simultaneously second order, exact dissipation of $E$ for every step size, a monotone interpolant, linear implicitness, and one softmax evaluation per step. Second, well-posedness is uniform in $\beta$; the linear system of Theorem~\ref{thm:sav} involves $\beta$ only through the bounded vector $b^{k}$, whereas second-order convex splitting schemes of secant type require an explicit stabilization proportional to the curvature bound $\tfrac{\beta}{2}\sigma_{\max}^2(X)$ of Lemma~\ref{lem:structure}(b) and hence degrade in the large-$\beta$ retrieval regime; this uniformity is the reason we adopt SAV rather than stabilized splitting as the reference second-order scheme. Third, the per-step cost equals that of one attention step, so the accuracy gain of order two is obtained at no cost in the accounting unit of the next section.

% --- end of Section 5 ---

% Section 6 -- draft v1
% Assumes: amsmath, amsthm, natbib; environments theorem, lemma, proposition, corollary, remark numbered within section.
% Depends on: Section 2 (def:schemes, prop:family), Section 3 (thm:local-family, thm:local-ie, lem:concentration, eq:retrieval-cond, rem:two-limits), Section 5 (thm:sav, prop:lte).

\section{Cost per softmax evaluation}\label{sec:cost}

The results so far rank the schemes by contraction per step, but the steps of different schemes have different prices, and Remark~\ref{rem:two-limits} already announced that the ranking changes when the accounting unit is normalized. This section fixes the unit, the evaluation
\[
x \;\longmapsto\; Xp(x) = X\operatorname{softmax}(\beta X^{\top}x),
\]
whose cost is two matrix--vector products with $X$ plus a length-$N$ softmax, about $4Nd$ floating point operations for a single query and a pair of GEMMs for a batch of queries, and expresses every scheme's convergence in error reduction per evaluation. All other arithmetic in the schemes of this paper is $O(d)$ or $O(N)$ per step and is neglected. Throughout this section the retrieval condition \eqref{eq:retrieval-cond} is in force on a ball $\bar{B}(\xi_\mu, r)$, $\ell = \ell_\mu(r)$, and rates are per evaluation in the sense of the factor by which the distance to the attractor $x^\ast_\mu$ is guaranteed to shrink each time $Xp(\cdot)$ is computed.

\subsection{The inexact implicit hierarchy}\label{sec:hierarchy}

The implicit scheme is computed in practice by an inner solver, and the first observation is that the resulting inexact schemes are not new objects but members of the structures already analyzed.

\begin{proposition}[One inner iteration is convex splitting]\label{prop:one-picard}
Let the implicit step \eqref{eq:ie} be approximated by $m$ iterations of the inner map $T_{x}(y) = (x + \Delta t\, Xp(y))/(1+\Delta t)$ of Theorem~\textup{\ref{thm:local-ie}}, warm started at $y_0 = x^{k}$. For $m = 1$ the resulting scheme is exactly the convex splitting scheme \eqref{eq:cs}, that is, $\Psi_{\theta}$ with $\theta = \Delta t/(1+\Delta t)$; for $m \to \infty$ it converges to the exact implicit scheme. The inexact implicit family thus interpolates between the damped attention family and the deep equilibrium limit; each inner iteration evaluates $Xp(\cdot)$ once and mixes it with the frozen anchor $x^{k}$, $T_{x^{k}}(y) = (1-\tau)\,x^{k} + \tau\, Xp(y)$ with $\tau = \Delta t/(1+\Delta t)$, and only the first, warm-started iteration coincides with a member of the family $\Psi_\theta$.
\end{proposition}

\begin{proof}
With $y_0 = x^{k}$, one application of $T_{x^{k}}$ gives $y_1 = (x^{k} + \Delta t\, Xp(x^{k}))/(1+\Delta t)$, which is \eqref{eq:cs}; the convergence of $y_m$ to the implicit step is Theorem~\ref{thm:local-ie}.
\end{proof}

\begin{proposition}[Closed-form certified factors of the Picard hierarchy]\label{prop:matched}
Under the assumptions of Theorem~\ref{thm:local-ie}, write $\tau = \Delta t/(1+\Delta t)$ and $\kappa = \kappa_{\Delta t}$, and let $\hat{\Phi}_{m}$ denote the inexact implicit step computed by $m$ warm-started inner iterations as in Proposition~\ref{prop:one-picard}. Then $\hat{\Phi}_{m}$ is Lipschitz on $\bar{B}(\xi_\mu, r)$ with certified factor $a_m$ obeying
\[
a_0 = 1, \qquad a_{m+1} = (1-\tau) + \tau\ell\, a_m, \qquad\text{whence}\qquad a_m = \kappa + (1-\kappa)\,(\tau\ell)^{m},
\]
strictly decreasing in $m$ from $a_1 = 1 - \tau(1-\ell)$, the convex splitting factor $q_{\tau}$ of Theorem~\ref{thm:local-family}, to $a_\infty = \kappa$, the local-branch factor of Theorem~\ref{thm:local-ie}. Moreover $a_m \ge \ell^{m}$ for every $m \ge 0$. Consequently $m$ evaluations spent inside one inexact implicit step certify a contraction no better than $\ell^{m}$, so that, normalized per softmax evaluation, no member of the warm-started Picard hierarchy carries a better certified factor than the attention bound $q_1 = \ell$ of Theorem~\ref{thm:local-family}. The comparison concerns certified worst-case bounds within this hierarchy; no claim is made about other inner solvers or about algorithm-level lower bounds.
\end{proposition}

\begin{proof}
Fix $x, x'$ in the ball with inner sequences $y_j, y'_j$ warm started at $y_0 = x$, $y'_0 = x'$. Subtracting the inner maps $T_x, T_{x'}$ gives $\|y_{j+1} - y'_{j+1}\| \le (1-\tau)\,\|x - x'\| + \tau\ell\,\|y_j - y'_j\|$, which is the stated recurrence for the certified Lipschitz factors; its fixed point is $(1-\tau)/(1-\tau\ell) = \kappa$, and solving the affine recurrence yields $a_m = \kappa + (1-\kappa)(\tau\ell)^{m}$, strictly decreasing since $\tau\ell < 1$ and $\kappa < 1$. The lower bound follows by induction: $a_0 = 1 = \ell^{0}$, and $a_{m+1} - \ell^{m+1} \ge (1-\tau) + \tau\ell^{m+1} - \ell^{m+1} = (1-\tau)\,(1 - \ell^{m+1}) \ge 0$. The per-evaluation statement is $a_m^{1/m} \ge \ell$.
\end{proof}

Proposition~\ref{prop:matched} replaces the outer/inner bookkeeping of earlier drafts by a single closed formula and makes the interpolation of Proposition~\ref{prop:one-picard} quantitative; the certified factors of the inexact implicit hierarchy slide monotonically from convex splitting to the local branch as inner iterations are added, and, normalized per softmax evaluation, they never improve on the attention bound $\ell$, because larger $\Delta t$ improves $\kappa$ only at the price of the inner iterations needed to approach it. This is a statement about certified worst-case bounds, not an algorithmic lower bound; a Newton inner solver, whose per-iteration cost is dominated by one action of the Hessian $\nabla^2E = I - \beta X H(p) X^{\top}$, itself two products with $X$ plus either matrix-free iterations or a Woodbury solve of an $N\times N$ system at $O(N^3 + dN)$, converges locally quadratically, and a per-evaluation analysis of that regime is not attempted here. The practical conclusion we do draw is the one anticipated in Remark~\ref{rem:two-limits}: within the well-separated regime the relaxed family at $\theta_\star$ carries the best certified per-evaluation factor among the schemes analyzed in this paper, improving on the attention bound $\ell$ by the modest factor $1/(2-\ell)$, an improvement that Campaign~G shows to be certified-tight yet empirically unrealized whenever the certificate is slack; and the use cases of the implicit machinery are those outside that regime, moderate $\beta$, queries near basin boundaries, and stiff generalized models in the sense of Section~\ref{sec:generalized}, where dissipation without step size restriction is the property purchased.

\subsection{Accounting summary}\label{sec:accounting}

The per-step and per-evaluation costs of the schemes analyzed in this paper are as follows; memory is beyond the storage of $X$ itself. The family $\Psi_\theta$ uses one evaluation and $O(d)$ memory per step, with certified per-evaluation factor $q_\theta$ of Theorem~\ref{thm:overrelax}, minimized at $\theta_\star = 2/(2-\ell)$ where it equals $\ell/(2-\ell)$. The inexact implicit scheme uses $m$ evaluations per outer step and $O(d)$ memory with Picard, or $O(dN)$ transient memory with Woodbury--Newton, and its certified factor after $m$ inner evaluations is $a_m = \kappa + (1-\kappa)(\tau\ell)^{m} \ge \ell^{m}$ by Proposition~\ref{prop:matched}. The SAV--CN scheme of Theorem~\ref{thm:sav} uses one evaluation, one Sherman--Morrison solve of cost $O(d)$, and the storage of one extra state and one scalar, so it delivers second-order trajectory accuracy at the per-evaluation price of attention, which by Proposition~\ref{prop:lte} makes it, together with ETD, the scheme of choice whenever the trajectory rather than the endpoint is the object of interest. For batched retrieval of $q$ queries all evaluations become GEMMs of shape $(q\times d)(d\times N)$ and the accounting is unchanged. We do not model reduced precision here.

% --- end of Section 6 ---

% Section 7 -- draft v1
% Assumes: amsmath, amsthm, natbib; environments theorem, lemma, proposition, corollary, remark, definition, assumption numbered within section.
% Depends on: Section 2 (lem:majorization, prop:family), Section 3 (thm:dissipation, thm:single-limit), Section 4 (thm:cells, lem:interpolation), Section 5 (thm:sav).

\section{Perspective: damped difference-of-convex iterations and non-quadratic convex parts}\label{sec:generalized}

The collapse of the explicit schemes into the family $\Psi_\theta$ rested on the convex part of \eqref{eq:energy} being exactly $\tfrac12\|x\|^2$. The associative memory literature has moved past this case; sparse models replace the log-sum-exp by conjugates of Tsallis entropic regularizers \citep{hu2023sparse,santos2024hfy}, and normalized models absorb $\ell_2$- and layer normalization into the energy \citep{santos2024hfy,hoover2023}. The natural generalization of the family, underrelaxation of the difference-of-convex step in dual coordinates, is not ours; a Bregman-divergence reading of DCA was developed by \citet{faust2023}, and \citet{niu2026} has shown that classical DCA is the full-step explicit Euler discretization of a dual-coordinate dynamics, introduced precisely the damped, Bregman-regularized DCA that generalizes $\Psi_\theta$, and established monotone descent, Kurdyka--{\L}ojasiewicz convergence and rates for it. We therefore confine this section to what transfers from Sections~\ref{sec:dissipation} and \ref{sec:basins} to that setting, recording one exact identity, one corollary on overrelaxation, and the basin-preservation statement, which is the part not covered by the optimization literature; a full treatment, including the local theory in the Hessian metric of the convex part, is deferred to a companion paper.

\begin{assumption}\label{ass:dc}
$E = E_{+} - E_{-}$ on $\mathbb{R}^d$, where \textup{(i)} $E_{+}$ is strictly convex, differentiable, and of Legendre type, with $\nabla E_{+}$ a bijection onto the interior of $\operatorname{dom} E_{+}^{\ast}$ and continuous inverse $\nabla E_{+}^{\ast}$; \textup{(ii)} $E_{-}$ is convex and differentiable with $\nabla E_{-}$ taking values in a bounded set $\mathcal{K} \subseteq \operatorname{ran}\nabla E_{+}$; \textup{(iii)} $E$ is coercive. $D_{+}(y,x) = E_{+}(y) - E_{+}(x) - \langle \nabla E_{+}(x), y-x\rangle$ denotes the Bregman divergence of $E_{+}$.
\end{assumption}

For the energy \eqref{eq:energy}, $E_{+} = \tfrac12\|\cdot\|^2$ and everything below reduces to Sections~\ref{sec:framework}--\ref{sec:basins}. In the sparse models $\nabla E_{-}(x) = X\pi(\beta X^{\top}x)$ with $\pi$ the sparsemax or $\alpha$-entmax map; at the sparsemax endpoint $\nabla E_{-}$ is Lipschitz continuous but not analytic; the elementary arguments below require only continuity of $\nabla E_{\pm}$. In the normalized models the convex part carries the constraint through the indicator of the \emph{ball} $\{q : \|q\| \le r\}$, whose conjugate machinery produces the $\ell_2$-normalization, and of the ball intersected with an affine subspace for layer normalization \citep{santos2024hfy}; the sphere itself is not convex, and Assumption~\ref{ass:dc}(i) fails as stated for these nonsmooth convex parts, so their treatment requires a subdifferential extension that we defer.

\begin{definition}[Damped DCA in dual coordinates \citep{niu2026}]\label{def:mirror-family}
For $\theta \in (0,1]$, define $\Psi^{+}_\theta$ by $\nabla E_{+}\big(\Psi^{+}_\theta(x)\big) = (1-\theta)\,\nabla E_{+}(x) + \theta\,\nabla E_{-}(x)$, well defined under Assumption~\ref{ass:dc}. For $\theta = 1$ this is one exact DCA step, $\nabla E_{+}^{\ast}\circ\nabla E_{-}$, the generalized attention update of the model at hand; for quadratic $E_{+}$ the family reduces to \eqref{eq:family}.
\end{definition}

\begin{theorem}[Exact Bregman dissipation identity]\label{thm:bregman-descent}
Under Assumption~\ref{ass:dc}, define also $D_{-}(y,x) = E_{-}(y) - E_{-}(x) - \langle \nabla E_{-}(x),\, y-x\rangle \ge 0$. Then for every $\theta \in (0,1]$ and every $x$, the update $x^{+} = \Psi^{+}_\theta(x)$ satisfies the identity
\begin{equation}\label{eq:bregman-descent}
E(x^{+}) - E(x) \;=\; -\,\frac{1}{\theta}\, D_{+}\big(x,\, x^{+}\big) \;-\; \frac{1-\theta}{\theta}\, D_{+}\big(x^{+},\, x\big) \;-\; D_{-}\big(x^{+},\, x\big),
\end{equation}
whose right-hand side is a sum of nonpositive terms. Descent inequalities for damped DCA appear in \citet{niu2026} and, for $\theta = 1$, in \citet{faust2023}; we record \eqref{eq:bregman-descent} because it is exact, and because for quadratic $E_{+}$ it reads $E(\Psi_\theta(x)) = E(x) - \tfrac12\theta(2-\theta)\|\nabla E(x)\|^2 - D_{-}(\Psi_\theta(x), x)$, identifying the slack of the certified decrease \eqref{eq:descent-family} as exactly the Bregman remainder of the concave part.
\end{theorem}

\begin{proof}
By the definition of $D_{-}$, $E(x^{+}) - E(x) = E_{+}(x^{+}) - E_{+}(x) - \langle \nabla E_{-}(x),\, x^{+}-x\rangle - D_{-}(x^{+}, x)$, with no inequality used. Solving the defining relation for $\nabla E_{-}(x) = \tfrac{1}{\theta}\nabla E_{+}(x^{+}) - \tfrac{1-\theta}{\theta}\nabla E_{+}(x)$ and substituting, then using the three-point identities $\langle \nabla E_{+}(x^{+}), x^{+}-x\rangle = E_{+}(x^{+}) - E_{+}(x) + D_{+}(x, x^{+})$ and $\langle \nabla E_{+}(x), x^{+}-x\rangle = E_{+}(x^{+}) - E_{+}(x) - D_{+}(x^{+}, x)$, the coefficient of $E_{+}(x^{+}) - E_{+}(x)$ is $1 - \tfrac{1}{\theta} + \tfrac{1-\theta}{\theta} = 0$, and what remains is \eqref{eq:bregman-descent}.
\end{proof}

\begin{corollary}[Overrelaxation under bounded asymmetry]\label{cor:overrelax}
Fix $x$ and $\theta>1$, and suppose the extrapolated dual point is admissible,
\[
 (1-\theta)\nabla E_{+}(x)+\theta\nabla E_{-}(x)\in\operatorname{ran}\nabla E_{+},
\]
so that $x^{+}$ is defined by the same dual relation as in Definition~\ref{def:mirror-family}. Suppose further that the resulting pair satisfies
\[
 D_{+}(x^{+},x)\le\kappa\,D_{+}(x,x^{+})
\]
for some $\kappa\ge1$. Then $E(x^{+})\le E(x)$, with strict decrease off critical points, whenever $\theta<1+1/\kappa$. A sufficient condition for the displayed asymmetry bound is that $E_{+}$ be $m$-strongly convex with $L$-Lipschitz gradient on a convex region containing the segment $[x,x^{+}]$, in which case one may take $\kappa=L/m$. In the quadratic case $\kappa=1$ globally, giving the full window $\theta\in(0,2)$.
\end{corollary}

\begin{proof}
The derivation of the identity \eqref{eq:bregman-descent} is valid for every $\theta > 0$; for $\theta > 1$, discarding the nonpositive term $-D_{-}(x^{+},x)$ gives $E(x^{+}) \le E(x) - \tfrac{1}{\theta}D_{+}(x,x^{+}) + \tfrac{\theta-1}{\theta}D_{+}(x^{+},x)$; bounding the last term by $\kappa\,\tfrac{\theta-1}{\theta}D_{+}(x,x^{+})$ yields $E(x^{+}) \le E(x) - \tfrac{1-\kappa(\theta-1)}{\theta}\,D_{+}(x,x^{+})$, which is strict decrease for $\theta < 1 + 1/\kappa$ unless $x^{+} = x$, that is, unless $x$ is critical.
\end{proof}

\begin{theorem}[Cell preservation for damped DCA]\label{thm:cells-general}
Under Assumption~\ref{ass:dc}, fix $\theta \in (0,1]$ and define the mirror segment $y_t = \nabla E_{+}^{\ast}\big((1-t\theta)\nabla E_{+}(x) + t\theta\,\nabla E_{-}(x)\big)$ for $t \in [0,1]$, a continuous path from $x$ to $\Psi^{+}_\theta(x)$ along which $E(y_t) \le E(x)$. Consequently the cell preservation Theorem~\ref{thm:cells} holds verbatim for the damped DCA family; for every local minimizer $x^\ast$ and every finite $c$ with $E(x^\ast)<c\le c^\ast(x^\ast)$, the energy cell $U_c(x^\ast)$ is contained in $\mathcal{B}_{\Psi^{+}_\theta}(x^\ast)$ for all $\theta \in (0,1]$. To our knowledge the basin statement is not contained in \citet{niu2026} or \citet{faust2023}, whose analyses concern descent and convergence rather than the geometry of attraction.
\end{theorem}

\begin{proof}
The point $y_t$ is $\Psi^{+}_{t\theta}(x)$ with $t\theta \in (0,1]$, so Theorem~\ref{thm:bregman-descent} gives $E(y_t) \le E(x)$, and $t \mapsto y_t$ is continuous by Assumption~\ref{ass:dc}(i); the component argument of the proof of Theorem~\ref{thm:cells} therefore confines the orbit, together with its mirror interpolant, to the bounded cell, so the orbit is precompact. Summing the identity \eqref{eq:bregman-descent} gives $D_{+}(x^{k}, x^{k+1}) \to 0$; since $E_{+}$ is strictly convex and $D_{+}$ is continuous on the compact closure of the cell, this forces $x^{k+1} - x^{k} \to 0$, for otherwise a subsequence with $\|x^{k+1} - x^{k}\| \ge \delta > 0$ would produce limits $\bar{x} \ne \bar{y}$ with $D_{+}(\bar{x}, \bar{y}) = 0$, contradicting strict convexity. If $x^{\dagger}$ is an accumulation point, passing to the limit in the defining relation of $\Psi^{+}_\theta$ along the subsequence, using $x^{k+1} - x^{k} \to 0$ and continuity of $\nabla E_{\pm}$ and $\nabla E_{+}^{\ast}$, yields $\nabla E_{+}(x^{\dagger}) = \nabla E_{-}(x^{\dagger})$, so $x^{\dagger}$ is critical; the open-ball argument of Theorem~\ref{thm:cells} places it in the cell, where the only critical point is $x^\ast$, and precompactness upgrades the unique accumulation point to convergence, $x^{k} \to x^\ast$. No analyticity or Kurdyka--{\L}ojasiewicz property is used.
\end{proof}

\begin{corollary}[Overrelaxed cell preservation]\label{cor:cells-overrelax}
Let $x^\ast$ be a local minimizer and let $c$ be finite with $E(x^\ast)<c\le c^\ast(x^\ast)$. Assume, in addition to Assumption~\ref{ass:dc}, that on the closure $\overline{U_c(x^\ast)}$ the asymmetry bound $D_{+}(y,x) \le \kappa\,D_{+}(x,y)$ holds for all $x,y$, as it does globally with $\kappa \le L/m$ when $E_{+}$ is $m$-strongly convex with $L$-Lipschitz gradient, and that the dual segment is admissible, $(1-s)\nabla E_{+}(x) + s\,\nabla E_{-}(x) \in \operatorname{ran}\nabla E_{+}$ for all $s \in [0,\theta]$ and all $x \in \overline{U_c(x^\ast)}$. Then Theorem~\ref{thm:cells-general} extends to every $\theta \in (0,\, 1 + 1/\kappa)$. Both hypotheses are imposed on the cell and its closure only, a set fixed in advance.
\end{corollary}

\begin{proof}
Let $x \in U_c(x^\ast)$ and let $y_t = \Psi^{+}_{t\theta}(x)$, $t \in [0,1]$, be the mirror segment, well defined and continuous by admissibility. The set $S = \{t \in [0,1] : y_s \in U_c(x^\ast) \text{ for all } s \le t\}$ is nonempty, since $y_0 = x$, and open in $[0,1]$ because $U_c(x^\ast)$ is open and $t \mapsto y_t$ continuous. It is also closed, as if $t_n \uparrow t$ with $t_n \in S$, then for each $n$ both $x$ and $y_{t_n}$ lie in $\overline{U_c(x^\ast)}$, where the hypotheses hold, so Theorem~\ref{thm:bregman-descent} for $t_n\theta \le 1$ and Corollary~\ref{cor:overrelax} for $t_n\theta > 1$ give $E(y_{t_n}) \le E(x) < c$; continuity yields $E(y_t) \le E(x) < c$, and $y_t$ lies in the same connected component of $\{E < c\}$ as $x$, that is, in $U_c(x^\ast)$. By connectedness of $[0,1]$, $S = [0,1]$; the interpolant is monotone and stays in the cell, so no circularity arises, and the orbit is confined to the bounded cell. Boundedness of the iterates now follows from this confinement, the convex-combination argument of the proof of Theorem~\ref{thm:cells-general} being unavailable for $\theta > 1$. Summing the decrease of Corollary~\ref{cor:overrelax} gives $\tfrac{1-\kappa(\theta-1)}{\theta}\sum_k D_{+}(x^{k}, x^{k+1}) < \infty$, so $D_{+}(x^{k}, x^{k+1}) \to 0$, and the remaining steps of the proof of Theorem~\ref{thm:cells-general}, separation by strict convexity, criticality of accumulation points by passage to the limit in the defining relation, and the open-ball argument, apply verbatim.
\end{proof}

We close with the finite-step side. The generalized Eyre splitting $\tfrac{1}{\Delta t}(x^{k+1}-x^{k}) = -\nabla E_{+}(x^{k+1}) + \nabla E_{-}(x^{k})$ remains well posed and unconditionally dissipative for every $\Delta t > 0$ by the same two-line convexity argument as in Theorem~\ref{thm:bregman-descent}, and tends, as $\Delta t \to \infty$, to the exact DCA step; for finite $\Delta t$ it differs from every member of the damped family unless $E_{+}$ is quadratic, so that the algebraic collapse of Proposition~\ref{prop:family} is a quadratic phenomenon even though dissipation and cells are not. Line-search overrelaxations of the DCA step have been studied as boosted DCA \citep{aragon2018bdca}, complementary to the certified window of Corollary~\ref{cor:overrelax}. The remaining components of a full theory, the local contraction analysis in the Hessian metric of $E_{+}$, the sharp margins of the sparse maps for which exact finite-time convergence is available \citep{santos2024hfy}, the nonsmooth normalized models, and a second-order construction replacing Section~\ref{sec:sav}, are the subject of the companion paper.

\section{Numerical experiments}\label{sec:experiments}

Nine computational campaigns test the results of Sections~\ref{sec:dissipation}--\ref{sec:cost} at small scale, each designed around a falsifiable prediction of a specific theorem: sharpness of the dissipation inequality (Theorem~\ref{thm:dissipation}), validity and slack of the local contraction factors (Theorems~\ref{thm:local-family} and \ref{thm:local-ie}), the error constants and orders of Proposition~\ref{prop:lte} and Theorem~\ref{thm:sav}, the tunneling threshold of Proposition~\ref{prop:tunneling}, the localization of basin discrepancies above the escape energy predicted by Theorem~\ref{thm:cells}, and the failure of true-energy dissipation for SAV anticipated after Theorem~\ref{thm:sav}. All computations are in IEEE double precision with a single fixed seed; the accompanying scripts, plain NumPy and SciPy, reproduce every number and figure of this section and are archived with the manuscript.

\subsection{Setup}\label{sec:exp-setup}

The quantitative campaigns (A, B, C, F) use $d = 64$, $N = 4$ patterns given by $\sqrt{d}$ times orthonormal random columns, so $M = \sigma_{\max}(X) = 8$ and $\Delta_\mu = 64$, with $\beta = 0.15$; this leaves the contractive regime, $\nu = \tfrac{\beta}{2}\sigma_{\max}^2 - 1 = 3.8 > 0$, and the retrieval condition \eqref{eq:retrieval-cond} holds at every pattern, so the model is genuinely multistable, while keeping the retrieval-ball quantities of Lemma~\ref{lem:concentration} at measurable size, $\varepsilon_\mu(1) = 2.24\times 10^{-3}$ and $\ell_\mu(1) = 4.30\times 10^{-2}$ on the ball of radius $r = 1$. Campaign D uses a one-dimensional two-pattern model where the global proximal step can be resolved to high numerical precision by global scalar minimization, and campaign E a two-dimensional three-pattern model where basins can be resolved on a grid.

\subsection{Campaign A: sharpness of the dissipation inequality}\label{sec:exp-A}

For $2{,}000$ points sampled uniformly in the ball of radius $2M$ and $\theta \in \{0.1, 0.2, \dots, 1.9\}$, the ratio of the observed decrease $E(x) - E(\Psi_\theta(x))$ to the certified decrease $\tfrac12\theta(2-\theta)\|\nabla E(x)\|^2$ of Theorem~\ref{thm:dissipation}(a) was computed for every sample. The minimum ratio over all $38{,}000$ pairs rounds to $1.0000$ at four decimal places throughout $\theta \le 1.6$, rising to $1.0004$ at $\theta = 1.9$; the inequality is never violated and is numerically sharp across the entire window, including near its endpoint, confirming that the unit-curvature majorization of Lemma~\ref{lem:majorization} captures the worst-case curvature of \eqref{eq:energy} exactly rather than up to constants; by the identity noted after Theorem~\ref{thm:dissipation}, the observed slack equals the Bregman remainder $D_{-}$ of the log-sum-exp along the step, strictly positive except for steps in $\ker X^{\top}$, where the log-sum-exp is flat and equality is exact; the ratio therefore exceeds one by an amount below the rounding shown.

\subsection{Campaign B: local contraction factors}\label{sec:exp-B}

The attractor $x^\ast_1$ was computed to machine precision; its distance to the pattern is $1.88\times 10^{-3}$, inside the certified bound $2M\varepsilon_\mu(1) = 3.58\times 10^{-2}$ of Theorem~\ref{thm:local-family}(b). Table~\ref{tab:contraction} compares the certified contraction factors with the worst observed per-step factors over $20$ random initializations at radius $0.8$ and up to $25$ steps each. Every bound holds; the family and implicit bounds are tight to within one to three percent except at the attention endpoint $\theta = 1$, where the certified $q_1 = \ell_\mu(r) = 4.30\times 10^{-2}$ overestimates the observed $2.61\times 10^{-3}$ by the slack accumulated in the constant $2$ of Lemma~\ref{lem:concentration}(iii) and in taking the supremum of $1 - p_\mu$ over the whole ball; the empirical spectral norm $\|\beta X H(p(x^\ast_1))X^{\top}\|_2 = 2.61\times 10^{-3}$ shows that essentially all of the slack is in the ball supremum, not in the linearization.

\begin{table}[t]
\centering
\begin{tabular}{lcc@{\qquad}lcc}
\toprule
scheme & certified & observed & scheme & certified & observed \\
\midrule
$\Psi_{0.25}$ & $0.7608$ & $0.7500$ & IE, $\Delta t = 0.25$ & $0.8069$ & $0.8000$ \\
$\Psi_{0.5}$  & $0.5215$ & $0.5002$ & IE, $\Delta t = 1$    & $0.5110$ & $0.5000$ \\
$\Psi_{1}$    & $0.0430$ & $0.0026$ & IE, $\Delta t = 4$    & $0.2071$ & $0.2001$ \\
              &          &          & IE, $\Delta t = 16$   & $0.0613$ & $0.0588$ \\
\bottomrule
\end{tabular}
\caption{Certified contraction factors $q_\theta = 1-\theta(1-\ell_\mu(r))$ and $\kappa_{\Delta t} = (1+\Delta t(1-\ell_\mu(r)))^{-1}$ against worst observed per-step factors near $x^\ast_1$ ($d=64$, $N=4$, $\beta=0.15$, $r=0.8$).}
\label{tab:contraction}
\end{table}

\subsection{Campaign C: orders and error constants}\label{sec:exp-C}

Trajectories from a point at distance $0.9$ from the attractor were integrated to $T = 4$ and compared with an RK4 reference at step $10^{-3}$. Table~\ref{tab:order} reports global errors and empirical convergence orders. All three family members converge at order one and SAV at order two, as asserted by Proposition~\ref{prop:lte} and Theorem~\ref{thm:sav}; the quantitative content is in the constants. At $\Delta t = 0.4$ the ETD error is $5.06\times 10^{-6}$ against $1.10\times 10^{-2}$ for EE, a factor of $2.2\times 10^{3}$, of the order of $1/\ell_\mu$, confirming that the ETD error constant is carried entirely by the softmax curvature as computed in Proposition~\ref{prop:lte}; on this trajectory the first-order ETD scheme is more accurate than the second-order SAV scheme at every tested step size, a reminder that near well-separated attractors the asymptotic order is not the operative figure of merit.

\begin{table}[t]
\centering
\begin{tabular}{lccccc c}
\toprule
 & $\Delta t=0.4$ & $0.2$ & $0.1$ & $0.05$ & $0.025$ & EOC \\
\midrule
EE  & $1.10\times10^{-2}$ & $6.11\times10^{-3}$ & $3.18\times10^{-3}$ & $1.62\times10^{-3}$ & $8.17\times10^{-4}$ & $0.85\to0.99$ \\
CS  & $1.46\times10^{-2}$ & $6.99\times10^{-3}$ & $3.40\times10^{-3}$ & $1.68\times10^{-3}$ & $8.31\times10^{-4}$ & $1.07\to1.01$ \\
ETD & $5.06\times10^{-6}$ & $2.36\times10^{-6}$ & $1.14\times10^{-6}$ & $5.62\times10^{-7}$ & $2.79\times10^{-7}$ & $1.10\to1.01$ \\
SAV & $2.23\times10^{-3}$ & $5.34\times10^{-4}$ & $1.33\times10^{-4}$ & $3.32\times10^{-5}$ & $8.32\times10^{-6}$ & $2.06\to2.00$ \\
\bottomrule
\end{tabular}
\caption{Global trajectory errors at $T=4$ against an RK4 reference, and empirical orders of convergence across the step-size range.}
\label{tab:order}
\end{table}

\subsection{Campaign D: the tunneling window}\label{sec:exp-D}

In the one-dimensional model with patterns $\xi_1 = 1$, $\xi_2 = -1.3$ and $\beta = 3$, the attractors sit at $0.9976$ and $-1.2997$ with energies $-0.5003$ and $-0.8450$, so $\delta = 0.3447$, and the query $x^0 = 0.8$ lies in the shallow cell. Proposition~\ref{prop:tunneling} certifies tunneling of the global proximal step for $\Delta t > \Delta t^\ast(x^0) = 6.40$; the empirical onset, located to high numerical precision by global scalar minimization of the proximal functional over a logarithmic sweep in $\Delta t$, is $5.36$. The bound is thus valid and conservative by seventeen percent in this configuration, and the certified-preservation threshold $1/\nu = 0.33$ leaves the window $[0.33,\ 5.36]$ in which, as anticipated after Proposition~\ref{prop:tunneling}, neither result applies; across this window the computed global step remained in the shallow cell, suggesting that the true tunneling threshold is governed by the geometry of the proximal envelope rather than by $1/\nu$, a gap the theory leaves open.

\subsection{Campaign E: basin discrepancies live above the escape energy}\label{sec:exp-E}

This campaign tests the falsifiable content of Theorem~\ref{thm:cells} directly against the flow, with no proxy. For the two-dimensional three-pattern model, the reference basins were computed by integrating \eqref{eq:flow} with an adaptive Runge--Kutta method at relative and absolute tolerance $10^{-11}$ from every node of a $301\times301$ grid over $[-3,3]^2$ to $T = 60$; a rerun at tolerance $10^{-9}$ on a random subsample of $400$ nodes reproduced every basin label, and every node converged. The critical points located by damped Newton iteration from a $41\times41$ multistart, a census that is numerical rather than certified, were classified by the Hessian: three minima, at energies $-2.00$, $-1.63$ and $-1.75$; three index-one saddles, at energies $-0.770$, $-1.145$ and $-0.869$; and one maximum. An escape level was inferred numerically for each attractor from the saddle values by a sublevel-component connectivity test, giving $\widehat{c}^{\,\ast} = -1.145$ for the two attractors separated by the lowest saddle and $\widehat{c}^{\,\ast} = -0.869$ for the third; these rest on the numerical census above and are not certified. The stable manifolds of the three saddles, integrated backward from the saddles along their stable eigenvectors, are the numerically resolved continuous separatrices, and Figure~\ref{fig:basins} overlays them, together with the level sets $E = \widehat{c}^{\,\ast}$, on the basin partitions of the flow and of $\Psi_{1.9}$.

Discrete basins were then computed for $\Psi_\theta$ with $\theta \in \{0.5, 1.0, 1.5, 1.9\}$ on the same grid, every node converging in every case, and each disagreeing node was located both in energy and in distance to the nearest point of the continuous separatrices. The disagreement fractions with the flow are $3.3\times10^{-5}$, $1.1\times10^{-4}$, $5.3\times10^{-3}$ and $0.135$ respectively; negligible throughout $\theta \le 1$, appreciable at $\theta = 1.5$, and massive at $\theta = 1.9$, in agreement with the approach to the collapse threshold of Proposition~\ref{prop:collapse}. The prediction of Theorem~\ref{thm:cells} is that no node can disagree while lying inside a cell below the escape energy of the attractor that the flow assigns to it; the count of such violations, evaluated node by node against the attractor-specific numerically inferred escape energy $\widehat{c}^{\,\ast}$, is \emph{zero} at every $\theta$, over $12{,}727$ disagreeing nodes in total. The geometry of the disagreements matches the qualitative picture of Remark~\ref{rem:frontier}; for $\theta \le 1.5$ they hug the separatrices, with median distance $0.003$, $0.007$ and $0.033$ and maximum distance $0.02$, $0.03$ and $0.14$ (the grid spacing is $0.02$), whereas at $\theta = 1.9$ they spread to median distance $1.19$ from the separatrices, that is, deep into the continuous basins, and are then excluded only by the escape-energy criterion, exactly as the theory permits. Together with the exact certificate agreement of Campaigns A and G, this is the sharpest test of the paper's central claim: the certified core is respected without exception, and every discrepancy is confined to the region the theory leaves open.

\begin{figure}[t]
\centering
\includegraphics[width=\textwidth]{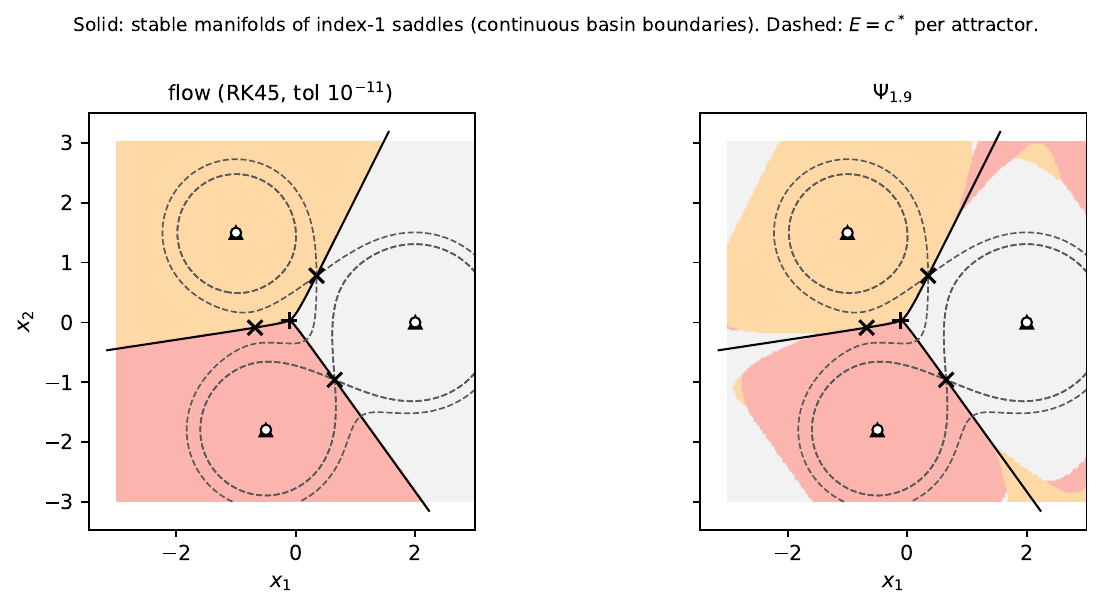}
\caption{Campaign E. Basins of attraction of the three attractors for the flow, computed by adaptive Runge--Kutta integration at tolerance $10^{-11}$ (left), and for $\Psi_{1.9}$ (right), on a $301\times301$ grid; triangles mark the patterns, circles the minima, crosses the index-one saddles, plus the maximum. Solid curves are the stable manifolds of the saddles, the numerically resolved continuous separatrices; dashed curves are the numerically inferred level sets $E = \widehat{c}^{\,\ast}$ of each attractor. Every basin discrepancy between the two panels lies outside the dashed level set of the attractor the flow assigns to it, as Theorem~\ref{thm:cells} requires.}
\label{fig:basins}
\end{figure}

\subsection{Campaign F: SAV dissipates the modified energy, not the true one}\label{sec:exp-F}

Finally, the caveat recorded after Theorem~\ref{thm:sav} is not an artifact of the proof. For the scheme of Definition~\ref{def:sav} with $\gamma = \tfrac12$, over ten random initializations at radius $1.5M$ and $200$ steps each, the modified energy $\tilde{E}^{k}$ obeyed the discrete law \eqref{eq:sav-law} to within $1.7\times10^{-14}$ per step at every tested $\Delta t$, machine-precision confirmation of the exact identity. The true energy $E(x^{k})$, by contrast, increased along the iteration in a nonnegligible fraction of steps at every step size, with worst per-step increases of $1.8\times10^{-3}$ at $\Delta t = 0.5$, $1.6\times10^{-2}$ at $\Delta t = 1$, and $3.2$ at $\Delta t = 2$. The order-one violation at $\Delta t = 2$, exceeding $3$ in energy units, shows that at large steps the auxiliary variable can decouple from $\sqrt{E_1^{\gamma} + C_0}$ by an amount that is not perturbative, so that SAV's unconditional stability, while real, certifies a genuinely different functional; for retrieval applications this means SAV should be operated at accuracy-limited rather than stability-limited step sizes, and it gives practical content to the combination problem stated in Section~\ref{sec:sav}: second order together with exact dissipation of the true energy, a monotone interpolant, linear implicitness, and one softmax evaluation per step.

\subsection{Campaign G: certified versus oracle overrelaxation}\label{sec:exp-G}

Theorem~\ref{thm:overrelax} and Remark~\ref{rem:oracle} were tested on the configuration of Section~\ref{sec:exp-setup} with $\beta$ swept over $\{0.09, 0.10, 0.12, 0.15, 0.20, 0.25\}$ at $r = 1$, spanning $\ell_\mu(r)$ from $0.46$ down to $5.9\times10^{-4}$. Table~\ref{tab:overrelax} reports, for each $\beta$, the certificate $\ell$, the attractor spectral radius $\lambda_{\max}$, the sampled ball supremum $\bar\lambda$ of $\|A\|$, the certified and observed factors at the certified optimum $\theta_\star$, and the mean iteration counts to tolerance $10^{-10}$ at $\theta = 1$, at $\theta_\star$, and at the oracle $\theta_{\mathrm{o}} = 2/(2-\lambda_{\max})$. Three predictions are confirmed. The observed factor at $\theta = 1$ equals $\lambda_{\max}$ to three digits, and the smallest eigenvalue of $A(x^\ast_1)$ vanishes to machine precision, as forced by the rank bound of Remark~\ref{rem:oracle}. The observed factor at $\theta_\star$ agrees with the certified $q_{\theta_\star} = \ell/(2-\ell)$ to three or four digits at every $\beta$, since the overrelaxed certificate is exactly tight, the kernel modes realizing $|1-\theta_\star|$. And the oracle factor equals $\lambda_{\max}/(2-\lambda_{\max})$ within a few percent, an improvement of at most a factor of two over attention worth at most one iteration anywhere in the sweep. The consequence anticipated in Remark~\ref{rem:oracle} follows; because, throughout the sweep, the slack $\ell/\lambda_{\max}$ never drops below $2(N-1) = 6$, the certified overrelaxation, although optimal among certified bounds and exactly tight, is slower than plain attention in observed iterations at every tested $\beta$, for instance $18.7$ versus $7.6$ iterations at $\beta = 0.09$ and $6.0$ versus $3.9$ at $\beta = 0.15$. Overrelaxation pays only when the certificate tracks the true ball supremum within a factor $2-\ell$; with the present constants it does not, and attention, sitting one certified notch below the optimum, is empirically unbeatable within the family in this regime.

\begin{table}[t]
\centering
\begin{tabular}{lcccccccc}
\toprule
$\beta$ & $\ell$ & $\lambda_{\max}$ & $\bar\lambda$ & $q_{\theta_\star}$ (cert = obs) & $\theta_\star$ & it($1$) & it($\theta_\star$) & it($\theta_{\mathrm{o}}$) \\
\midrule
$0.09$ & $0.460$ & $7.69\times10^{-2}$ & $8.54\times10^{-2}$ & $0.2984$ & $1.298$ & $7.6$ & $18.7$ & $7.0$ \\
$0.10$ & $0.316$ & $4.40\times10^{-2}$ & $5.45\times10^{-2}$ & $0.1877$ & $1.188$ & $6.8$ & $14.0$ & $6.0$ \\
$0.12$ & $0.145$ & $1.44\times10^{-2}$ & $2.04\times10^{-2}$ & $0.0783$ & $1.078$ & $5.1$ & $9.0$ & $5.0$ \\
$0.15$ & $4.30\times10^{-2}$ & $2.61\times10^{-3}$ & $3.85\times10^{-3}$ & $0.0220$ & $1.022$ & $3.9$ & $6.0$ & $4.0$ \\
$0.20$ & $5.20\times10^{-3}$ & $1.41\times10^{-4}$ & $2.18\times10^{-4}$ & $2.61\times10^{-3}$ & $1.003$ & $3.0$ & $4.0$ & $3.0$ \\
$0.25$ & $5.90\times10^{-4}$ & $7.20\times10^{-6}$ & $1.47\times10^{-5}$ & $2.95\times10^{-4}$ & $1.000$ & $2.0$ & $3.0$ & $2.0$ \\
\bottomrule
\end{tabular}
\caption{Campaign G. Certified versus observed overrelaxation across the $\beta$ sweep. Iteration counts are means over ten initializations at radius $0.8\rho$ to tolerance $10^{-10}$; at every $\beta$ the observed factor at $\theta_\star$ reproduces the certified $\ell/(2-\ell)$ to the digits shown.}
\label{tab:overrelax}
\end{table}

\subsection{Campaign H: correlated patterns and the certificate--retrieval gap}\label{sec:exp-H}

All quantitative campaigns so far use orthogonal patterns with $N \ll d$, the most favourable geometry. Campaign~H sweeps correlated patterns $\xi_\mu \propto \sqrt{1-\rho}\,g_\mu + \sqrt{\rho}\,g_0$, normalized to $\|\xi_\mu\| = \sqrt d$, with correlation $\rho \in \{0, 0.3, 0.6\}$, ratios $N/d \in \{1/16, 1/2, 2\}$ including the overcomplete case $N > d$, three inverse temperatures per ratio, and three seeds per configuration, reporting worst cases over seeds. Two questions are asked of every configuration. First, is the unconditional dissipation inequality of Theorem~\ref{thm:dissipation}(a) respected and sharp? Over $27$ configurations, $\theta \in \{0.5,1,1.5,1.9\}$ and $500$ random points each, the minimum ratio of observed to certified decrease rounds to $1.0000$ in every configuration, with the largest observed slack $2\times10^{-3}$ at $(d,N,\rho) = (64,32,0.6)$, $\beta = 0.15$; the majorization of Lemma~\ref{lem:majorization} is exact irrespective of correlation, aspect ratio, or $\beta$, as its proof, which uses none of these, predicts. Second, how do the certified retrieval condition \eqref{eq:retrieval-cond} and empirical retrieval compare? Table~\ref{tab:corr} reports, per configuration, the fraction of patterns satisfying \eqref{eq:retrieval-cond} at $r = 0.05\sqrt d$ and the fraction retrieved empirically, defined by an attractor within $0.05\sqrt d$ of the pattern with observed local factor $\lambda_{\max}(A(x^\ast)) < 0.5$.

\begin{table}[t]
\centering\small
\begin{tabular}{lcc@{\;\;}ccc@{\;\;}ccc}
\toprule
& & & \multicolumn{3}{c}{certified fraction} & \multicolumn{3}{c}{empirical fraction} \\
$(d,N)$ & $\rho$ & $\beta$ triple & & & & & & \\
\midrule
$(64,4)$ & $0$   & $0.15/0.3/0.6$ & $1.00$ & $1.00$ & $1.00$ & $1.00$ & $1.00$ & $1.00$ \\
$(64,4)$ & $0.3$ & $0.15/0.3/0.6$ & $0.00$ & $1.00$ & $1.00$ & $1.00$ & $1.00$ & $1.00$ \\
$(64,4)$ & $0.6$ & $0.15/0.3/0.6$ & $0.00$ & $0.00$ & $1.00$ & $0.00$ & $1.00$ & $1.00$ \\
$(64,32)$ & $0$   & $0.15/0.3/0.6$ & $0.00$ & $1.00$ & $1.00$ & $1.00$ & $1.00$ & $1.00$ \\
$(64,32)$ & $0.3$ & $0.15/0.3/0.6$ & $0.00$ & $0.06$ & $1.00$ & $0.16$ & $1.00$ & $1.00$ \\
$(64,32)$ & $0.6$ & $0.15/0.3/0.6$ & $0.00$ & $0.00$ & $0.00$ & $0.00$ & $0.88$ & $1.00$ \\
$(16,32)$ & $0$   & $0.5/1/2$ & $0.00$ & $0.06$ & $0.81$ & $0.25$ & $1.00$ & $1.00$ \\
$(16,32)$ & $0.3$ & $0.5/1/2$ & $0.00$ & $0.00$ & $0.00$ & $0.00$ & $0.41$ & $1.00$ \\
$(16,32)$ & $0.6$ & $0.5/1/2$ & $0.00$ & $0.00$ & $0.00$ & $0.00$ & $0.00$ & $0.31$ \\
\bottomrule
\end{tabular}
\caption{Campaign H. Worst case over three seeds of the fraction of patterns satisfying the certified retrieval condition \eqref{eq:retrieval-cond} at $r = 0.05\sqrt d$, and of the fraction retrieved empirically, across correlation $\rho$, aspect ratio $N/d$ and inverse temperature $\beta$. Certified retrieval was accompanied by empirical retrieval in every configuration; the converse fails increasingly with correlation and overcompleteness.}
\label{tab:corr}
\end{table}

Three conclusions follow. Certified retrieval was accompanied by empirical retrieval in every configuration, despite the empirical criterion $\lambda_{\max}(A(x^\ast)) < 0.5$ being stricter than the certified bound $\lambda_{\max} \le \ell < 1$, and where both hold the certificate slack $\ell/\lambda_{\max}$ ranges from $5.6$ in the orthogonal $(64,4)$ case at $\beta = 0.15$ to above $10^{4}$ in overcomplete cases, consistent with Campaign~G; in our experiments the slack never drops below $2(N-1)$. The certified condition fails well before empirical retrieval does; at $\rho = 0.3$, $(64,4)$, $\beta = 0.15$, no pattern is certified but all four are retrieved with median error $7\times10^{-3}$ and median local factor $0.064$, and in the overcomplete regime $(16,32)$ at $\rho = 0.3$, $\beta = 2$, retrieval is perfect while nothing is certified. The two mechanisms behind the loss of certification are visible in the table's covariates: correlation reduces the minimum separation $\Delta_\mu/\|\xi\|^2$ from $0.82$ to $0.31$ at $\rho = 0.6$ for $(64,4)$, and both correlation and overcompleteness inflate $\sigma_{\max}^2(X)/d$ from $1.3$ to above $24$, which enters $\ell_\mu(r)$ linearly and $\nu$ through the uniqueness threshold. Finally, empirical retrieval itself degrades and eventually fails, at $\rho = 0.6$ for small $\beta$ and throughout the most correlated overcomplete row, where the median retrieval error exceeds $0.35\sqrt d$; the certified condition is therefore not vacuously conservative but tracks a genuine failure of the memory at a factor of two to four in $\beta$. Since Theorems~\ref{thm:cells} and \ref{thm:dissipation} carry no separation hypothesis, their guarantees are unaffected by correlation; what correlation and overcompleteness erode is the explicit certification of retrieval balls in Section~\ref{sec:local}, whose constant $\sigma_{\max}^2(X)$ is the natural target for sharpening in exactly the regime, $N > d$ with correlated patterns, where modern Hopfield layers operate.

\subsection{Campaign I: the discrete-gradient benchmark}\label{sec:exp-I}

Section~\ref{sec:sav} named the midpoint discrete gradient of \citet{gonzalez1996} as the exact-dissipation benchmark against which SAV should be read; Campaign~I runs it. With $\bar\nabla E(x,y) = \nabla E(\tfrac{x+y}{2}) + \big(E(y) - E(x) - \langle \nabla E(\tfrac{x+y}{2}), y-x\rangle\big)\,(y-x)/\|y-x\|^2$, which satisfies $\langle \bar\nabla E(x,y), y-x\rangle = E(y) - E(x)$ identically, the scheme $x^{k+1} = x^{k} - \Delta t\,\bar\nabla E(x^{k}, x^{k+1})$ obeys $E(x^{k+1}) - E(x^{k}) = -\tfrac{1}{\Delta t}\|x^{k+1}-x^{k}\|^2$ exactly, for the true energy and for every $\Delta t$; the implicit step was solved by a hybrid Newton method with finite-difference Jacobian, and cost is counted with one softmax-gradient evaluation as unit and one energy evaluation, a single product with $X^{\top}$ plus a log-sum-exp, as half a unit. On the trajectory of Campaign~C the scheme is second order to the digits, with empirical orders $2.00$ across all four step-size halvings and global errors $8.8\times10^{-4}$, $2.2\times10^{-4}$, $5.5\times10^{-5}$, $1.4\times10^{-5}$ and $3.4\times10^{-6}$ at $\Delta t = 0.4$ down to $0.025$, about two and a half times smaller than SAV at each step; the exact-dissipation identity holds to $1.4\times10^{-14}$ per step. On the setup of Campaign~F, ten initializations at radius $1.5M$ and $200$ steps, where SAV raised the true energy by up to $1.8\times10^{-3}$, $1.6\times10^{-2}$ and $3.2$ at $\Delta t = 0.5$, $1$ and $2$, the discrete gradient shows worst increases of $2.1\times10^{-14}$, $2.8\times10^{-14}$, $4.3\times10^{-14}$ and, at $\Delta t = 4$, $9.6\times10^{-13}$, with identity residuals below $2\times10^{-12}$ throughout; roundoff, not dissipation failure. The property is bought at a price. Away from the attractors the fixed-point (Picard) inner solver does not converge at large steps, and the generic Newton solver spent between $200$ and $300$ evaluation units per step, against $1$ for SAV and for every member of the relaxed family; a Newton solver with the analytic Jacobian of $\bar\nabla E$, which costs a few Hessian actions per iteration, would reduce this substantially but not to the one-evaluation cost of the linearly implicit schemes. The comparison thus places the two second-order schemes at opposite ends of the same trade-off announced in Section~\ref{sec:sav}; SAV, one evaluation and an exact law for a modified energy that the true energy can violate by order one at large steps; the discrete gradient, an exact law for the true energy at every step and, in our generic reference implementation, a solver cost two orders of magnitude higher, a gap that analytic Jacobians or Jacobian-free Newton--Krylov solvers would reduce substantially and that we do not claim intrinsic. Neither is known to admit a monotone interpolant in the sense of Section~\ref{sec:basins}, and the combination problem stated there stands.

\subsection{Reproducibility}\label{sec:exp-repro}

Campaigns A--F run in under two minutes on a single CPU core from one self-contained script with fixed seed $20260809$; Campaigns G, H and I are separate scripts with the same seed conventions and run in a few minutes each. The accompanying scripts reproduce every numerical table and figure in this article, and the raw outputs needed to verify the reported values are included in the reproducibility archive. Persistent repository identifiers are given in the Data and Code Availability statement below.

\section{Conclusions}\label{sec:conclusions}

The question this paper poses to a time discretization of modern Hopfield retrieval dynamics is not whether it dissipates the energy or preserves the equilibria, but whether it preserves the assignment of queries to memories. Theorem~\ref{thm:cells} answers it for an explicit and elementary part of each basin; every energy cell below the escape energy of its attractor, that is, every connected component of a sublevel set containing that attractor and no other critical point, lies simultaneously in the basin of the flow, in the basin of every member of the relaxed attention family $\Psi_\theta$ with $\theta \in (0,2)$, and in the basin of the implicit Euler relation throughout its uniqueness regime. Energy cells are common certified basin cores, and the mechanism is a single structural fact, the unit-curvature majorization of Lemma~\ref{lem:majorization}, which holds uniformly in the inverse temperature and yields both the unconditional dissipation of Theorem~\ref{thm:dissipation} and a monotone continuous interpolant of every discrete orbit; the argument inside a cell is elementary and needs no {\L}ojasiewicz-type input. Two mechanisms delimit the statement: the global proximal map tunnels out of non-global cells beyond an explicit step threshold, at which point non-global local minimizers cease to satisfy $x \in \mathcal{P}_{\Delta t}(x)$, and the explicit Euler map turns the attractor into a repeller beyond $\theta \approx 2$.

Two lessons from the certified local theory deserve emphasis because they cut against a naive reading. First, the local certificate $\ell_\mu(r)$ of Lemma~\ref{lem:concentration} is conservative; it exceeded the observed spectral radius by a factor of at least $2(N-1)$ throughout our experiments, and by more than four orders of magnitude for correlated overcomplete patterns, so that certified retrieval fails well before retrieval itself does. Second, and as a consequence, certified optimality and empirical performance can point in opposite directions; the overrelaxation $\theta_\star = 2/(2-\ell)$ of Theorem~\ref{thm:overrelax} minimizes the certified contraction factor, and its bound is exactly attained when $N \le d$, yet plain attention converged faster in observed iterations at every inverse temperature we tested, precisely because the certificate is slack. We regard reporting both sides of this gap as part of the result rather than as a caveat to it.

Four problems remain open. The first is quantitative rather than qualitative; above the escape energy our results localize the possible discrepancies between continuous and discrete basins but do not bound them, and a measure-theoretic or Hausdorff estimate of the deformation as a function of $\theta$ and $\beta$ would complete the picture; the unstable fixed points recently shown to be attached to faces of the pattern polytope \citep{beise2026} are the natural organizing objects. The second is the combination problem of Section~\ref{sec:sav}; discrete-gradient methods dissipate the true energy exactly at second order but are fully implicit, the SAV scheme is linearly implicit at one softmax evaluation per step but dissipates only a modified energy, and no scheme known to us achieves second order, exact dissipation of $E$ for every step size, a monotone interpolant, and one evaluation per step. The third is sharpening the certificate itself, whose constant $\sigma_{\max}^2(X)$ degrades exactly in the correlated overcomplete regime where Hopfield layers operate. The fourth is the nonsmooth side of Section~\ref{sec:generalized}; Corollary~\ref{cor:cells-overrelax} closes the smooth Bregman extension by preserving cells throughout a certified overrelaxed window, but normalized retrieval models place the constraint inside the convex part through the indicator of a ball, so that $\nabla E_{+}$ becomes a subdifferential and the mirror family a normalize-after-mix update. Extending the cell-preservation theorem to that nonsmooth setting appears to us the most valuable single extension left open here.

\section*{Acknowledgments}
This work was supported by the Universidad de M\'alaga, Spain, through project PPRO-B4-2026-004 of its Plan Propio de Investigaci\'on, Transferencia y Divulgaci\'on Cient\'ifica.

\section*{Data and Code Availability}
The source code, reproducibility scripts, numerical outputs, and figure-generation scripts supporting this study are publicly available at Zenodo, DOI: 10.5281/zenodo.22048179, and are mirrored at the GitHub repository FrancisRVillatoro/modern-hopfield-basin-discretizations.

\section*{Competing Interests}
The author declares no competing interests.

\section*{AI-Assisted Preparation Disclosure}
Anthropic Claude Opus 5 and OpenAI ChatGPT was used during manuscript preparation for language editing, mathematical consistency checks, and LaTeX editing. The author independently reviewed and verified the mathematical statements, numerical results, citations, and final text and takes full responsibility for the content.

% --- end of Section 8 ---

% ---------- appendices -----------------------------------------------------
\appendix

\section{Genericity for two patterns (proof of Lemma~\ref{lem:genericity})}\label{app:barrier}

Let $N = 2$ and let $\xi_1, \xi_2 \in \mathbb{R}^d$ be linearly independent (so $d \ge 2$); write $v = \xi_1 - \xi_2 \ne 0$ and $p(x) = (p_1, p_2)$ with $p_1 p_2 = e^{\beta v^{\top}x}/(1 + e^{\beta v^{\top}x})^2 > 0$ for every $x$. Then $H(p) = p_1 p_2 \begin{psmallmatrix} 1 & -1 \\ -1 & 1 \end{psmallmatrix}$, so $X H(p) X^{\top} = p_1 p_2\, v v^{\top}$ and
\[
Df(x)\, f(x) \;=\; -f(x) \;+\; \beta\, p_1 p_2\, \langle v, f(x)\rangle\, v ,
\]
which lies in $\operatorname{span}\{f(x)\}$ if and only if $f(x) \perp v$ or $f(x) \parallel v$. Suppose, for contradiction, that this dichotomy holds at every $x \in \mathbb{R}^d$. The sets $\{x : \langle v, f(x)\rangle = 0\}$ and $\{x : f(x) \wedge v = 0\}$ are closed and cover $\mathbb{R}^d$, so by the Baire category theorem at least one contains a nonempty open set. If $\langle v, f\rangle \equiv 0$ on an open set $O$, differentiating gives $Df(x)^{\top} v = 0$ on $O$, and since $Df^{\top} = Df = -I + \beta p_1 p_2\, vv^{\top}$ this reads $v = \beta p_1 p_2 \|v\|^2\, v$, forcing $p_1 p_2 \equiv 1/(\beta\|v\|^2)$ on $O$; but $p_1 p_2$ is a strictly non-constant function of $v^{\top}x$ alone, and the open set $O$ contains segments in the direction $v$, a contradiction. If instead $f \wedge v \equiv 0$ on a nonempty open set, then, all components of $f \wedge v$ being real analytic on $\mathbb{R}^d$, they vanish identically, so $f(x) \parallel v$ for every $x$; evaluating at $x = 0$, where $p(0) = (\tfrac12, \tfrac12)$ independently of the patterns and of $\beta$, gives $f(0) = \tfrac12(\xi_1 + \xi_2) \ne 0$, and $(\xi_1 + \xi_2) \parallel (\xi_1 - \xi_2)$ holds if and only if $\xi_1 \wedge \xi_2 = 0$, that is, if and only if the patterns are linearly dependent, which is excluded by hypothesis. Hence a point $x_0$ with $Df(x_0)f(x_0) \notin \operatorname{span}\{f(x_0)\}$ exists, proving Lemma~\ref{lem:genericity}; the extension to $N \ge 3$ is left as the conjecture stated after Proposition~\ref{prop:barrier}. \qed

\section{Second-order convergence of the SAV--CN scheme}\label{app:sav}

Throughout, $\gamma \in (0,1)$ is fixed, $E_1^{\gamma}(x) = \tfrac{\gamma}{2}\|x\|^2 - \operatorname{lse}_\beta(X^{\top}x)$, $u(x) = E_1^{\gamma}(x) + C_0$ with $C_0 = M^2/(2\gamma) + \beta^{-1}\log N + 1$, and $b = \nabla E_1^{\gamma}/\sqrt{u}$ with $\nabla E_1^{\gamma}(x) = \gamma x - X p(x)$.

\paragraph{Bounds on $b$.} Since $\operatorname{lse}_\beta(X^{\top}x) \le M\|x\| + \beta^{-1}\log N$, completing the square gives $u(x) \ge \big(\sqrt{\gamma/2}\,\|x\| - M/\sqrt{2\gamma}\,\big)^2 + 1 \ge 1$ for every $x$, so $u$ grows quadratically at infinity, while $g(x):=\nabla E_1^{\gamma}(x)$ satisfies $\|g(x)\| \le \gamma\|x\| + M$. Hence $\|b(x)\| \le B_\gamma$ for some finite constant, with $\|b(x)\| \to \sqrt{2\gamma}$ as $\|x\| \to \infty$. Moreover
\[
 Db(x)=u(x)^{-1/2}\nabla^2E_1^{\gamma}(x)-\frac12u(x)^{-3/2}g(x)g(x)^{\top}.
\]
By Lemma~\ref{lem:structure}(b), $\|\nabla^2E_1^{\gamma}(x)\|\le\gamma+\tfrac{\beta}{2}\sigma_{\max}^2(X)$ globally. The first term in $Db$ is therefore bounded because $u\ge1$, and the second is bounded because $\|g(x)\|^2=O(\|x\|^2)$ whereas $u(x)^{3/2}=\Theta(\|x\|^3)$ at infinity. Thus $\sup_x\|Db(x)\|<\infty$, so $b$ is globally Lipschitz; write $L_b$ for a global Lipschitz constant. This is the regularity used below.

\paragraph{Consistency.} Insert the exact solution of the extended system $\dot{x} = -(1-\gamma)x - b(x)\,r$, $\dot{r} = \tfrac12\langle b(x), \dot{x}\rangle$, initialized on the manifold $r = \sqrt{u(x)}$, into \eqref{eq:sav}, and write $t^{k+\frac12} = (k+\tfrac12)\Delta t$. The trapezoidal averages satisfy $\tfrac12(x(t^{k+1})+x(t^{k})) = x(t^{k+\frac12}) + O(\Delta t^2)$ and $\tfrac12(r(t^{k+1})+r(t^{k})) = r(t^{k+\frac12}) + O(\Delta t^2)$; the divided differences equal the midpoint derivatives up to $O(\Delta t^2)$ by the midpoint expansion; and the frozen coefficient satisfies $b(\bar{x}^{k+\frac12}) = b(x(t^{k+\frac12})) + O(\Delta t^2)$ because the linear extrapolation $\tfrac32 x(t^{k}) - \tfrac12 x(t^{k-1})$ agrees with $x(t^{k+\frac12})$ to $O(\Delta t^2)$ and $b$ is globally Lipschitz. Both defects are therefore $O(\Delta t^2)$ per unit step, uniformly on bounded time intervals. The startup contributes an error $O(\Delta t^2)$, since one ETD step has local truncation error $O(\Delta t^2)$ by Proposition~\ref{prop:lte}, and $r^{1} = \sqrt{u(x^{1})}$ is exact at $x^{1}$, so $\eta^{1} = O(\Delta t^{2})$ below.

\paragraph{Stability and convergence.} Let $e_x^{k} = x^{k} - x(t^{k})$, $e_r^{k} = r^{k} - r(t^{k})$ and $\eta^{k} = \|e_x^{k}\| + |e_r^{k}|$. Subtracting the defect equations from \eqref{eq:sav} and eliminating $e_r^{k+1}$ as in the solvability step of Theorem~\ref{thm:sav}, the update matrix $\mathcal{A}^{k} = (\tfrac{1}{\Delta t} + \tfrac{1-\gamma}{2}) I + \tfrac14\, b^{k} (b^{k})^{\top}$ satisfies $\mathcal{A}^{k} \succeq (\tfrac{1}{\Delta t} + \tfrac{1-\gamma}{2}) I$, so $\|(\mathcal{A}^{k})^{-1}\| \le \Delta t$; the difference of frozen coefficients obeys $\|b(\bar{x}^{k+\frac12}) - b(\bar{y}^{k+\frac12})\| \le L_b\,(\tfrac32\|e_x^{k}\| + \tfrac12\|e_x^{k-1}\|)$; and $r^{k}$ and $r(t^{k})$ are bounded uniformly on $[0,T]$ by the discrete law \eqref{eq:sav-law} and the continuous energy law, respectively. Collecting terms yields, for a constant $C$ depending on $\gamma$, $B_\gamma$, $L_b$, the data and $T$, but not on $\Delta t$,
\[
\eta^{k+1} \;\le\; (1 + C\,\Delta t)\,\eta^{k} \;+\; C\,\Delta t\,\eta^{k-1} \;+\; C\,\Delta t^{3},
\]
and the two-step discrete Gr\"onwall inequality with $\eta^{0} = 0$ and $\eta^{1} = O(\Delta t^{2})$ gives 
\[\max_{0 \le k \le T/\Delta t}\, \eta^{k} \le C_T\,\Delta t^{2}.
\] 
The scheme is thus unconditionally solvable and second-order convergent on bounded time intervals, with no step-size restriction and no truncation of the nonlinearity. \qed

% ---------- bibliography ---------------------------------------------------
\bibliographystyle{abbrvnat}
\bibliography{refs_v1}

\end{document}